\documentclass[reqno,11pt,twoside]{article}
\usepackage[hmargin=1.25in,vmargin=1.25in, a4paper, centering]{geometry}
\usepackage{fourier}
\usepackage{ebgaramond}
\usepackage{amsmath, amsthm, amssymb, bm}
\usepackage{graphicx}
\usepackage{xcolor}
\usepackage{enumerate}
\usepackage[format=hang, margin=20pt, font=small, labelfont=bf, labelsep=colon]{caption}
\usepackage{booktabs,array}
\usepackage[section]{placeins}
\numberwithin{equation}{section}
\theoremstyle{plain}
\newtheorem{thm}{Theorem}[section]
\newtheorem{prop}[thm]{Proposition}
\newtheorem{fact}[thm]{Fact}
\newtheorem{lem}[thm]{Lemma}

\theoremstyle{definition}
\newtheorem{defn}[thm]{Definition}

\theoremstyle{remark}
\newtheorem{rem}[thm]{Remark}
 \newcommand{\addQEDstyle}[2]{\AtBeginEnvironment{#1}{\pushQED{\qed}\renewcommand{\qedsymbol}{#2}}\AtEndEnvironment{#1}{\popQED}}
 \addQEDstyle{ex}{$\triangleleft$}
  \addQEDstyle{defn}{$\triangleleft$}
 \addQEDstyle{rem}{$\blacktriangleleft$}
\usepackage[noadjust]{marginnote}

\usepackage[toc,titletoc]{appendix}
\appendixtocoff
\usepackage{etoolbox, apptools}
\AtBeginEnvironment{appendices}{\appendixtrue}
\usepackage[nobottomtitles,pagestyles]{titlesec}
\usepackage{emptypage}
\widenhead*{0.5in}{0.5in}
\renewpagestyle{plain}[\bfseries]{\footrule\setfoot{}{\thepage}{}}
\newpagestyle{mystyle}[\bfseries]{
\headrule
\sethead[\thepage][][M. Capoferri and I. Mann]{Spectral asymptotics for linear elasticity: the case of mixed boundary conditions}{}{\thepage}
}
\renewcommand\footnotemark{}
\titleformat{\section}
{\normalfont\large\bfseries}
{\IfAppendix{\filcenter\appendixname~\thesection.}{\filcenter\S\thesection.}}{1ex}{\filcenter}
\titleformat{\subsection}[runin]
{\normalfont\normalsize\bfseries}{\S\thesubsection.}{1ex}{}[.]
\newcommand{\LL}{\mathcal{L}}
\newcommand{\TT}{\mathcal{T}}
\newcommand{\Ric}{\operatorname{Ric}}

\newcommand{\grad}{\operatorname{grad}}

\newcommand{\curl}{\operatorname{curl}}

\newcommand{\dr}{\mathrm{d}}

\newcommand{\bu}{\mathbf{u}}
\newcommand{\bv}{\mathbf{v}}

\newcommand{\Dir}{\mathrm{Dir}}
\newcommand{\free}{\mathrm{free}}

\newcommand{\prin}{\mathrm{prin}}

\usepackage[colorlinks,allcolors=blue,pagebackref]{hyperref}
\usepackage{pifont}
\renewcommand*{\backrefalt}[4]{%
\ifcase #1 %
No citations%
\or
\ding{43}~p.~#2%
\else
\ding{43}~pp.~#2%
\fi}

\usepackage{orcidlink}

\begin{document}

\allowdisplaybreaks

\title{Spectral asymptotics for linear elasticity:
\\ the case of mixed boundary conditions%
\footnote{{\bf MSC2020: }Primary 35P20. Secondary 35Q74, 74J05.}%
\footnote{{\bf Keywords: } elasticity, eigenvalue counting function, Dirichlet conditions, free boundary conditions, mixed boundary conditions.}%
}
\author{
Matteo Capoferri\,${}^{\orcidlink{0000-0001-6226-1407}}$\thanks{\textbf{MC}:
Maxwell Institute for Mathematical Sciences and
Department of Mathematics, Heriot-Watt University, Edinburgh EH14 4AS, UK; m.capoferri@hw.ac.uk,
\url{https://mcapoferri.com}.
}
\and
Isabel Mann\thanks{\textbf{IM}:
Department of Mathematics, University of York, York YO10 5DD, UK;
im921@york.ac.uk.
}
}

\renewcommand\footnotemark{}

\date{\today}

\maketitle
\begin{abstract}
We establish two-term spectral asymptotics for the operator of linear elasticity with mixed boundary conditions on a smooth compact Riemannian manifold of arbitrary dimension. We illustrate our results by explicit examples in dimension two and three, thus verifying our general formulae both analytically and numerically. 
\end{abstract}

{\small \tableofcontents}

\section{Introduction}
\label{Introduction}
The operator of linear elasticity is one of the fundamental operators of mathematical physics, describing the deformation of an (isotropic) elastic body. 
The main thrust of this paper is to derive an explicit formula for the second asymptotic term (often called \emph{second Weyl coefficient}) in the expansion of the eigenvalue counting function for the operator of linear elasticity with mixed boundary conditions on a smooth $d$-dimensional Riemannian manifold with boundary. This paper complements the analysis performed in \cite{2terms}, where the two cases of ``pure'' Dirichlet and free boundary conditions were examined.

\

The structure of the paper as follows. 

In subsections~\ref{The operator of linear elasticity} and~\ref{Boundary value problems} we introduce setting and notation, before stating the problem and our main results in subsection~\ref{Statement of the problem and main results}. 

Section~\ref{Proof of main theorem} is devoted to the proof of our main result, Theorem~\ref{main theorem}. The proof comes in several steps: first we present a streamlined version of the algorithm for the calculation of the second asymptotic term (subsection~\ref{A streamlined algorithm}); secondly, we reduce the problem at hand to the two-dimensional analogue plus a much simpler $(d-2)$-dimensional problem by identifying appropriate invariant subspaces (subsection~\ref{Invariant subspaces}); lastly, we prove our main result by implementing the algorithm in each invariant subspace separately (subsection~\ref{The proof}).

In section~\ref{Explicit examples} we examine explicit examples in dimensions two and three. This is an integral part of the paper which serves both as an illustration and a verification of our results. Remarkably, for two- and three-dimensional flat cylinders we write down the full spectrum of the operator of linear elasticity with mixed boundary conditions explicitly, and compute the two-term spectral asymptotics analytically.

\subsection{The operator of linear elasticity}
\label{The operator of linear elasticity}

Let $(M,g)$ be a compact connected smooth Riemannian manifold of dimension $d\ge 2$ with boundary $\partial M$.
We denote by $\nabla$ the Levi-Civita connection and by $\mathrm{Ric}$ the Ricci curvature tensor. 

We define the operator of linear elasticity $\LL$ acting on vector fields $\bu$ on $M$ as
\begin{equation}
\label{L}
(\LL \bu)^\alpha:=-\mu \left(\nabla_\beta \nabla^\beta u^\alpha+\Ric^\alpha{}_\beta u^\beta \right)-(\lambda+\mu)\nabla^\alpha\nabla_\beta u^\beta\,.
\end{equation}
Here and further on we adopt the Einstein summation convention over repeated indices. The quantities $\lambda$ and $\mu$ are real constants known as \emph{Lam\'e parameters}, assumed to satisfy the conditions
\begin{equation}
\label{conditions on lambda and mu}
\mu>0, \qquad d\lambda+2\mu>0\,,
\end{equation}
which guarantee strong convexity, see, e.g., \cite{miyanishi,agranovich}.
Furthermore, we assume that the material density of the of the elastic medium $\rho_\mathrm{mat}$ differs from the Riemannian density $\sqrt{\det g}$ by a constant positive factor.

\

The principal symbol $\LL_\prin$ of $\LL$ reads\footnote{Here and further on $\|\xi\|$ denotes the Riemannian norm of the covector $\xi$.}
\begin{equation}
\label{Lprin}
\left[\LL_\prin\right]^\alpha{}_\beta(x,\xi)=\mu \|\xi\|^2 \delta^\alpha{}_\beta+(\lambda+\mu)\xi^\alpha\xi_\beta\,,
\end{equation}
which, on account of \eqref{conditions on lambda and mu}, immediately implies that $\LL$ is elliptic.
Indeed, the eigenvalues of $\LL_\prin$ are
\begin{equation}
\label{eigenvalues of Lprin}
\mu\|\xi\|^2 \quad \text{(with multiplicity $d-1$)}, \qquad(\lambda+2\mu)\|\xi\|^2 \quad \text{(with multiplicity $1$)}\,.
\end{equation}
Clearly, the operator $\LL$ is formally self-adjoint with respect to the $L^2$ inner product 
\begin{equation*}
\label{inner product}
(\bu,\bv)_{L^2(M)}:=\int_M g_{\alpha\beta}\,u^\alpha v^\beta\, \sqrt{\det g}\,\dr x\,.
\end{equation*}

\subsection{Boundary value problems}
\label{Boundary value problems}

Consider the potential energy of elastic deformation
\begin{equation}
\label{E}
\mathcal{E}[\bu]:=\frac{1}{2}\int_M \left(\lambda(\nabla_\alpha u^\alpha)^2+\mu(\nabla_\alpha u_\beta+\nabla_\beta u_\alpha)\nabla^\alpha u^\beta \right) \sqrt{\det g}\,\dr x
\end{equation}
associated with the vector field of displacements $\bu$. The quadratic form $\mathcal{E}[\bu]$ is nonnegative for $\bu\in H^1(\Omega)$ and strictly positive for $\bu\in H^1_0(\Omega)$. Observe that the structure of the quadratic functional \eqref{E} of linear elasticity is the result of certain geometric assumptions, see \cite[formula (8.28)]{part1}, as well as \cite[Example 2.3 and formulae (2.5a), (2.5b) and (4.10e)]{diffeo}.

Performing integration by parts in \eqref{E} one obtains the Green identity for the elasticity operator 
\begin{equation}
\label{Green identity}
2\,\mathcal{E}[\bu]=(\bu, \LL \bu)_{L^2(M)}+(\bu, \TT \bu)_{L^2(\partial M)}\,,
\end{equation}
where $\TT$ is the boundary traction operator defined as
\begin{equation*}
\label{T}
(\TT\bu)^\alpha:=\lambda n^\alpha \nabla_\beta u^\beta  +\mu\left(n^\beta \nabla_\beta u^\alpha + n_\beta \nabla^\alpha u^\beta\right).
\end{equation*}
Here $\mathbf{n}$ is the exterior unit normal vector to the boundary $\partial M$.

Examination of \eqref{Green identity} supplies appropriate boundary conditions for $\LL$. In the current paper, we will be concerned with the following four sets of boundary conditions.
\begin{itemize}
\item Dirichlet boundary conditions:
\begin{equation}
\label{Dir bc}
\left.\bu\right|_{\partial M}=0.
\end{equation}
\item Free boundary conditions:
\begin{equation}
\label{free bc}
\left.\TT\bu\right|_{\partial M}=0.
\end{equation}
\item Dirichlet-free (DF) boundary conditions:
\begin{equation}
\label{DF bc}
\left.\left[\bu -\left(g_{\alpha\beta}\,n^\alpha u^\beta\right)\,\mathbf{n}\right]\right|_{\partial M}=0\,,
\qquad
\left.g_{\alpha\beta}\,n^\alpha (\TT\bu)^\beta\right|_{\partial M}=0\,.
\end{equation}
\item Free-Dirichlet (FD) boundary conditions:
\begin{equation}
\label{FD bc}
\left.\left[\TT\bu -\left(g_{\alpha\beta}\,n^\alpha (\TT u)^\beta\right)\,\mathbf{n}\right]\right|_{\partial M}=0\,,
\qquad
\left.g_{\alpha\beta}\,n^\alpha u^\beta\right|_{\partial M}=0\,.
\end{equation}
\end{itemize}

We refer to the boundary conditions DF and FD as \emph{mixed boundary conditions}. The former (DF) corresponds to Dirichlet boundary conditions being imposed tangentially to the boundary and free conditions imposed in the normal direction to the boundary; the latter (FD) corresponds to free boundary conditions being imposed tangentially to the boundary and Dirichlet conditions imposed in the normal direction to the boundary.

The boundary conditions \eqref{Dir bc}--\eqref{FD bc} are of Shapiro--Lopatinski type \cite{KruTuo} for $\LL$, hence the corresponding boundary value problems are elliptic. This leads to the following four eigenvalue problems for $\LL$. 

\begin{enumerate}
\item {\bf The Dirichlet problem (Dir)}. The \emph{Dirichlet eigenvalue problem} consists in seeking $\bu\in H^1(\Omega)$, $\bu \ne 0$, and $\Lambda \in \mathbb{R}$ such that
\begin{equation}
\label{L u=Lambda u}
\LL\bu=\Lambda \bu
\end{equation}
subject to the boundary conditions \eqref{Dir bc}. The problem \eqref{L u=Lambda u}, \eqref{Dir bc} has discrete spectrum, consisting of discrete eigenvalues
\begin{equation*}
(0<)\Lambda_1^\Dir\le \Lambda_2^\Dir\le \dots
\end{equation*}
enumerated with account of multiplicity and accumulating to $+\infty$.


\item {\bf The free boundary problem (free)}. The \emph{free boundary eigenvalue problem} consists in
seeking $\bu\in H^1(\Omega)$, $\bu \ne 0$, and $\Lambda \in \mathbb{R}$ satisfying \eqref{L u=Lambda u}, subject to the boundary conditions \eqref{free bc}. The problem \eqref{L u=Lambda u}, \eqref{free bc} has discrete spectrum, consisting of discrete eigenvalues
\begin{equation*}
(0\le)\Lambda_1^\free\le \Lambda_2^\free\le \dots
\end{equation*}
enumerated with account of multiplicity and accumulating to $+\infty$.

\item {\bf The Dirichlet-free problem (DF)}. The \emph{Dirichlet-free eigenvalue problem} consists in seeking $\bu\in H^1(\Omega)$, $\bu \ne 0$, and $\Lambda \in \mathbb{R}$ satisfying \eqref{L u=Lambda u}, subject to the boundary conditions \eqref{DF bc}. The problem \eqref{L u=Lambda u}, \eqref{DF bc} has discrete spectrum, consisting of discrete eigenvalues
\begin{equation*}
(0\le)\Lambda_1^{\mathrm{DF}}\le \Lambda_2^{\mathrm{DF}}\le \dots
\end{equation*}
enumerated with account of multiplicity and accumulating to $+\infty$.

\item {\bf The free-Dirichlet problem (FD)}. The \emph{free-Dirichlet eigenvalue problem} consists in seeking $\bu\in H^1(\Omega)$, $\bu \ne 0$, and $\Lambda \in \mathbb{R}$ satisfying \eqref{L u=Lambda u}, subject to the boundary conditions \eqref{FD bc}. The problem \eqref{L u=Lambda u}, \eqref{FD bc} has discrete spectrum, consisting of discrete eigenvalues
\begin{equation*}
(0\le)\Lambda_1^{\mathrm{FD}}\le \Lambda_2^{\mathrm{FD}}\le \dots
\end{equation*}
enumerated with account of multiplicity and accumulating to $+\infty$.
\end{enumerate}

\begin{rem}
The problems $\Dir$, $\free$, DF and FD also admit a minmax formulation. We refer the interested reader to \cite{LMS} for details.
\end{rem}

Let us briefly elaborate on the physical meaning of the above eigenvalue problems. The spectral parameter $\Lambda$ appearing in \eqref{L u=Lambda u} has the following interpretation
\begin{equation*}
\Lambda=\frac{\rho_\mathrm{mat}}{\sqrt{\det g}} \,\omega^2\,,
\end{equation*}
where $\omega$ is the angular natural frequency of oscillation of the elastic medium. The boundary conditions $\Dir$ \eqref{Dir bc} describe a body whose boundary is ``clamped'', i.e., completely prevented from moving, whereas the conditions $\free$ \eqref{free bc} describe the opposite situation, in which the boundary is free to oscillate without restrictions. Mixed DF boundary conditions describe a body that is allowed to deform in the direction normal to the boundary, but is prevented from deforming in the directions tangential to the boundary. Similarly, mixed FD boundary conditions describe a body that is allowed to ``slide'' along its boundary, but is prevented from deforming in the direction normal to the boundary. Clearly, mixed boundary conditions are physically meaningful and describe realistic scenarios relevant for applications --- see also \cite{LMS} for further discussions in this respect. 

\subsection{Statement of the problem and main results}
\label{Statement of the problem and main results}

Consider, for each set of boundary conditions $\aleph\in\{\Dir,\free, \mathrm{DF}, \mathrm{FD}\}$,
the corresponding \emph{eigenvalue counting function} $N_\aleph:\mathbb{R}\to \mathbb{N}$ defined as
\begin{equation}
\label{N aleph}
N_\aleph(\Lambda):=\# \left\{k\ | \ \Lambda_k^\aleph <\Lambda\right\}\,.
\end{equation}
Clearly, the function \eqref{N aleph} is monotonically non-decreasing in $\Lambda$ and vanishes identically for $\Lambda\le \Lambda_1^\aleph$.

The study of the asymptotic behaviour of eigenvalue counting functions of the type \eqref{N aleph} as $\Lambda\to+\infty$ for (semibounded) elliptic operators is a well established area of mathematics, pioneered by Lord Rayleigh's \emph{The theory of sound} \cite{Ra77} in 1877. What started as an investigation prompted by practical questions from physics soon attracted the interest of pure mathematicians, as people realised that the coefficients in these expansions contain geometric invariants (see, e.g., \cite[Chapter~6]{levitinbook}). We refer the reader to \cite{SaVa, ANPS, Ivr} for historical overviews of the development of the subject.

\

Before stating our main result, let us summarise, without proof, some known facts concerning \eqref{N aleph}. In what follows $(M,g)$ satisfies the conditions from subsection~\ref{The operator of linear elasticity}.

\begin{prop}
\label{proposition Weyl law}
We have
\begin{equation}
\label{Weyl law}
N_\aleph(\Lambda)=a \operatorname{Vol}_d(M) \Lambda^{d/2}+ o\left(\Lambda^{d/2}\right) \quad \text{as}\quad \Lambda \to +\infty,
\end{equation}
where
\begin{equation}
\label{coefficient a}
a=\frac{1}{(4\pi)^{d/2}\Gamma \left(1+\frac{d}2\right)}\left(\frac{d-1}{\mu^{d/2}}+\frac{1}{(\lambda+2\mu)^{d/2}} \right)
\end{equation}
is the \emph{Weyl constant} for linear elasticity, $\operatorname{Vol}_d(M)$ is the Riemannian volume of $M$, and $\Gamma$ is the gamma function.
\end{prop}

The one-term asymptotic expansion \eqref{Weyl law} is often referred to as \emph{Weyl law}. Note that in the special case $d=3$ formula \eqref{Weyl law} was already established, indirectly and on the basis of physical arguments, by P. Debye in 1912 \cite{Debye}. A rigorous mathematical proof was provided shortly afterwards by H.~Weyl \cite{We15}.
It is worth emphasising that the coefficient $a$ is independent of the choice of boundary conditions.

\

Let $\mathcal{A}$ be an elliptic semibounded differential operator of even order $2m$ acting between sections of Hermitian $C^\infty$ vector bundles of dimension $N$ over a smooth $d$-dimensional manifold $M$ with boundary, supplemented by differential boundary conditions $\mathcal{B}$ satisfying the (parabolic version of the) Shapiro--Lopatniski conditions \cite{AV64}. Then it is known \cite[Theorem~2.6.1]{greiner} that the trace of the Green kernel $G(x,y,t)$ for the boundary value problem
\begin{equation*}
\begin{cases}
\left(\frac{\partial}{\partial t}+\mathcal{A}\right)\bu=0 &\text{in }M,\\
\mathcal{B}\bu=0 &\text{on }\partial M,
\end{cases}
\end{equation*}
admits a complete asymptotic expansion 
\begin{multline}
\label{expansion for heat kernel}
\mathcal{Z}_\mathcal{B}(t):=\int_M \operatorname{tr} G(x,x,t) \,\operatorname{dVol}_M\sim \widetilde{c}_{d-1} \, t^{-\frac{d}{2m}}+ \widetilde{c}_{d-2} \, t^{-\frac{d}{2m}+\frac{1}{2m}}+\dots+\widetilde{c}_{d-k}t^{-\frac{d}{2m}+\frac{n-1}{2m}}+\dots
\\ \text{as}\quad t\to 0^+,
\end{multline}
where $\mathrm{tr}$ stands for matrix trace. Moreover, all coefficients in \eqref{expansion for heat kernel} are locally determined \cite[\S2]{greiner}. We refer the reader to \cite{Grubb99}, \cite[\S4.2]{Grubb_book} and references therein for further details and generalisations.

\

Formula \eqref{expansion for heat kernel} allows us to define Weyl coefficients for elliptic operators on manifolds with boundary.

\begin{defn}
\label{definition Weyl coefficients}
For $1\le n\le d$, we define the \emph{$n$-th Weyl coefficient} for the elliptic boundary value problem $(\mathcal{A},\mathcal{B})$ to be the number
\begin{equation}
\label{definition Weyl coefficients equation 1}
c_{d-n}:=\frac{\widetilde{c}_{d-n}}{\Gamma\left( \frac{d-n+1}{2m}\right)}\,,
\end{equation}
where the $\widetilde{c}_{d-n}$ is the coefficient of $t^{-\frac{d}{2m}+\frac{n-1}{2m}}$ in the expansion \eqref{expansion for heat kernel} and $\Gamma$ is the gamma function.
\end{defn}

\begin{rem}
Note that Definition~\ref{definition Weyl coefficients} agrees with the definition of Weyl coefficients in
\cite{ASV, wave, part2, diagonalisation, curl} for (pseudo)differential operators on compact manifolds without boundary, where Weyl coefficients are defined to be the coefficients appearing in the complete asymptotic expansion for the mollified derivative of the counting function. Whilst such a complete asymptotic expansion always exists when $\partial M=\emptyset$ (see, e.g., \cite{Ivr98}), we are unaware of a similar result for manifolds \emph{with} boundary. This is why in the latter case defining Weyl coefficients is somewhat more delicate.

We should also point out that the standard convention in the literature is to call Weyl coefficients the constants appearing in the asymptotic expansion of the mollified counting function, as opposed to its derivative. The two definitions are, effectively, the same up to integrating factors; as a matter of convenience and consistency with previous papers by the first author, we will stick here with Definition~\ref{definition Weyl coefficients}.
\end{rem}

\begin{fact}
Suppose that the eigenvalue counting function $N(\Lambda)$ of the elliptic eigenvalue problem 
\begin{equation*}
\begin{cases}
\mathcal{A}\bu=\Lambda \bu &\text{in }M,\\
\mathcal{B}\bu=0 &\text{on }\partial M
\end{cases}
\end{equation*}
admits a $j$-term asymptotic expansion
\begin{equation}
\label{general k-term expansion}
N(\Lambda)=C_{d}\Lambda^{\frac{d}{2m}}+C_{d-1}\Lambda^{\frac{d-1}{2m}}+ \dots +C_{d-j+1}\Lambda^{\frac{d-j+1}{2m}}+ o(\Lambda^{\frac{d-j+1}{2m}}) \quad \text{as}\quad \Lambda \to+\infty
\end{equation}
for some $1\le j \le d$. Then
\begin{equation}
\label{relation between C and c}
C_{d-n+1}=\frac{2m}{d-n+1}c_{d-n} \qquad\text{for}\quad 1\le n\le j\,.
\end{equation}
\end{fact}

\begin{rem}
We should like to emphasise that Weyl coefficients \eqref{definition Weyl coefficients equation 1} are defined --- and can be computed --- \emph{irrespective} of whether an asymptotic expansion for the eigenvalue counting function of the form \eqref{general k-term expansion} exists.
\end{rem}

\

Of course, when $\mathcal{A}=\LL$ and $\mathcal{B}=\mathcal{B}_\aleph$, $\aleph\in\{\mathrm{DF}, \mathrm{FD }\}$, formula \eqref{expansion for heat kernel} specialises to read
\begin{equation*}
\label{two-term elasticity}
\mathcal{Z}_\aleph(t)=\sum_{k=1}^\infty e^{-\Lambda_k^\aleph t}=\Gamma\left(\frac{d}2+1\right)\,a \operatorname{Vol}_d(M)t^{-\frac{d}{2}}+ c_{d-2}^\aleph \,t^{-\frac{d-1}{2}}+ o(t^{-\frac{d-1}{2}})\,\quad\text{as}\quad t\to 0^+,
\end{equation*}
where the first asymptotic term has been written explicitly in terms of \eqref{coefficient a} --- cf.~\eqref{Weyl law} and~\eqref{relation between C and c}.

\

Whilst the existence of a one-term asymptotic expansion (Weyl's law) is always guaranteed --- see Proposition~\ref{proposition Weyl law} --- the validity of a two-term expansion of the form~\eqref{general k-term expansion} for $\LL$ with boundary conditions $\mathcal{B}_\aleph$, $\aleph\in\{\Dir, \free, \mathrm{DF}, \mathrm{FD}\}$, is still an open problem. Nevertheless, such an expansion is known to exist under additional dynamical assumptions on certain \emph{branching Hamiltonian billiards} on the cotangent bundle $T^*M$. We recall the result below, referring the reader to \cite{Vas86} for additional details and precise statements.

\begin{thm}
\label{theorem two-term asymptotics with billiards conditions}
Suppose that $(M,g)$ is such that the corresponding billiards is neither absolutely periodic nor dead-end. Then 
\begin{equation}
\label{theorem two-term asymptotics with billiards conditions equation 1}
N_\aleph(\Lambda)=a \operatorname{Vol}_d(M) \Lambda^{d/2}
+
C_{d-1}^\aleph\Lambda^{(d-1)/2}
+ 
o\left(\Lambda^{(d-1)/2}\right) 
\quad
\text{as}
\quad 
\Lambda \to +\infty\,,
\end{equation}
for any set of boundary conditions $\aleph\in\{\Dir, \free, \mathrm{DF}, \mathrm{FD}\}$.
\end{thm}

Theorem~\ref{theorem two-term asymptotics with billiards conditions} is a special case of \cite[Theorem~6.1]{Vas84}, which is applicable here because the eigenvalues \eqref{eigenvalues of Lprin} of $\LL_\mathrm{prin}$ have constant multiplicity as functions of $(x,\xi)\in T^*M\setminus\{0\}$.

\

We are now ready to state our main result.

\begin{thm}
\label{main theorem}
Let $(M, g)$ be a smooth compact connected $d$-dimensional Riemannian manifold with boundary $\partial M$, $d\ge2$. The second Weyl coefficient for the elliptic boundary value problem $(\LL, \mathcal{B}_\aleph)$, $\aleph\in\{\mathrm{DF}, \mathrm{FD}\}$, is given by
\begin{equation}
\label{main theorem equation 0}
c^\aleph_{d-2}=\frac{d-1}2C^\aleph_{d-1}=\frac{d-1}2 b_{\aleph} \operatorname{Vol}_{d-1}(\partial M)\,,
\end{equation}
where
\begin{equation}
\label{main theorem equation 1}
b_{\aleph}
:=
@
\frac{1}{2^{d+1}\pi^{\frac{d-1}2} \Gamma\left(\frac{d+1}{2}\right)}
\left(
\frac{d-3}{\mu^{\frac{d-1}2}}
+
\frac1{(\lambda+2\mu)^{\frac{d-1}{2}}}
\right)
\quad \text{with}\quad
@=\begin{cases}
- & \text{for}\ \,\aleph=\mathrm{DF},\\
+ & \text{for}\ \,\aleph=\mathrm{FD}.
\end{cases}
\end{equation}
\end{thm}

The above theorem, whose proof will be given in Section~\ref{Proof of main theorem}, warrants a number of remarks.

\begin{enumerate}[(i)]
\item Formulae for $b_\aleph$ in the case of "pure" boundary conditions $\aleph\in\{\Dir, \free\}$ were obtained in \cite[Theorem~1.8]{2terms}\footnote{To ease the comparison, note that formulae (1.27) and (1.28) in \cite[Theorem~1.8]{2terms} are expressed in terms of the auxiliary quantity $\alpha:=\frac{\mu}{\lambda+2\mu}\in \left(0, \frac{d}{2(d-1)} \right)$.}.

\item Remarkably, formula \eqref{main theorem equation 1} is very simple and elegant. This is noteworthy and in general not the case for boundary value problems for linear elasticity, in that one would expect the Lam{\'e} parameters $\lambda$ and $\mu$ to mix up in a rather complicated way in the second Weyl coefficient, owing to boundary conditions mixing longitudinal and transverse waves. For instance, the expressions for $b_\Dir$ and $b_\free$ contain integrals of inverse trigonometric functions depending on $\alpha:=\frac{\mu}{\lambda+2\mu}$ in a nontrivial fashion,
see \cite[formulae (1.27) and (1.28)]{2terms}. The underlying reason for \eqref{main theorem equation 1} being so simple is that mixed boundary conditions $\mathrm{DF}$ and $\mathrm{FD}$, unlike $\Dir$ and $\free$, \emph{do not} mix up components of the vector fields they act upon when one switches to the associated one-dimensional spectral problem --- see~Section~\ref{A streamlined algorithm}, and formula~\eqref{B'} in particular.

\item
Note that
\begin{equation}
\label{FD less DF}
b_{FD}<0<b_{DF}
\end{equation}
for $d=2$, whereas
\begin{equation}
\label{DF less FD}
b_{DF}<0<b_{FD}
\end{equation}
for $d\ge 3$.

\item We should like to emphasise that computing the second Weyl coefficient for \emph{systems} of PDEs is not easy. Indeed, the subject area of two-term asymptotics for elliptic systems has experienced a troubled development, up until the last decade; we refer the reader to \cite[Section~11]{CDV} for a historical overview.
\end{enumerate}

Before moving on to the proof of our main theorem, let us recall a well known fact which will bring about some simplifications in subsequent sections.

\begin{fact}
\label{fact about flat space}
The first two Weyl coefficients do not feel the geometry of $M$ or of its boundary $\partial M$. Therefore, it suffices to determine these coefficients in the case where $M$ is a smooth domain in $\mathbb{R}^d$ equipped with the Euclidean metric.
\end{fact}

\section{Proof of Theorem~\ref{main theorem}}
\label{Proof of main theorem}

This section is concerned with the proof of Theorem~\ref{main theorem}. We will break the proof, somewhat long and technical, into several steps.

On account of Fact~\ref{fact about flat space}, in the remainder of this section we will assume that $M$ is a smooth domain in $\mathbb{R}^d$ and that $g$ is the Euclidean metric, $g_{\alpha\beta}=\delta_{\alpha\beta}$.

\subsection{A streamlined algorithm}
\label{A streamlined algorithm}

In order to prove Theorem~\ref{main theorem} we will rely on a constructive algorithm for the second Weyl coefficient due to Vassiliev \cite{Vas84}. The original results from \cite{Vas84} (see also~\cite{SaVa}) apply, strictly speaking, to scalar operators. A roadmap for the generalisation to systems was given in \cite[\S~6]{Vas84}, whereas a detailed exposition of the algorithm for systems was recently provided in \cite{2terms}. For the convenience of the reader, we present below a streamlined version of the latter algorithm, adapted to the special case of the operator of linear elasticity \eqref{L}.

\

In a neighbourhood of the boundary $\partial M$ we introduce local coordinates $x=(x',z)$, with $x'\in\mathbb{R}^{d-1}$ and $z:=\operatorname{dist}(x,\partial M)$ for $x\in \operatorname{Int} (M)$, so that $\partial M=\{z=0\}$ and $z>0$ inside of $M$. Similarly, we adopt the notation $\xi=(\xi',\zeta)\in \mathbb{R}^{d-1}\times \mathbb{R}$ and $\bu=(\bu', u_d)$. Furthermore, we denote by $\mathcal{B}_\aleph$ the differential operators implementing the boundary conditions \eqref{DF bc} and \eqref{FD bc} for $\aleph=\mathrm{DF}$ and 
$\aleph=\mathrm{FD}$, respectively.

\

Consider the \emph{one-dimensional spectral problem}
\begin{equation}
\label{one-dimensional spectral problem}
\LL'\bu(z)=\Lambda \bu(z), \qquad \left.\mathcal{B}_\aleph' \bu\right|_{z=0}=0\,, \qquad \aleph\in\{\mathrm{DF},\mathrm{FD}\},
\end{equation}
where $\LL'$ and $\mathcal{B}_\aleph'$ are the ordinary differential operators acting on vector functions $\bu=\bu(z)$ defined in accordance with
\begin{equation}
\label{L'}
\LL'\bu:=\mu\left(|\xi'|^2-\frac{\dr^2}{\dr z^2}\right) \bu-(\lambda+2\mu)\begin{pmatrix}
i \xi'\\
\frac{\dr}{\dr z}
\end{pmatrix} \left(i\xi'\cdot \bu'+\frac{\dr u_d}{\dr z} \right)
\end{equation}
and
\begin{equation}
\label{B'}
\left.\mathcal{B}'_\mathrm{DF}\bu(z)\right|_{z=0}
=
\begin{pmatrix}
\bu'(0)
\\
-(\lambda+2\mu)\frac{\dr u_d}{\dr z}(0)
\end{pmatrix},
\qquad
\left.\mathcal{B}'_\mathrm{FD}\bu(z)\right|_{z=0}
=
\begin{pmatrix}
-\mu \frac{\dr \bu'}{\dr z}(0)
\\
u_d(0)
\end{pmatrix}.
\end{equation}
The operators \eqref{L'} and~\eqref{B'} are obtained from $\LL$ and $\mathcal{B}_\aleph$, $\aleph\in\{\mathrm{DF},\mathrm{FD}\}$, by replacing partial derivatives along the boundary with $i$ times the corresponding component of momentum, $\partial_{x'}\mapsto i\xi'$. Furthermore, in the second component of $\mathcal{B}'_\mathrm{DF}\bu$ we dropped terms proportional to $\bu'(0)$, whereas in the first component of $\mathcal{B}'_\mathrm{FD}\bu$ we dropped terms proportional to $u_d(0)$.

\

Suppose $\xi'\ne 0$.

\

{\bf Step 1:} {\it Thresholds and continuous spectrum}. The principal symbol 
$
\LL'_\prin(\zeta)
$
of $\LL'$ (recall~\eqref{Lprin}) has eigenvalues 
\begin{equation}
\label{eigenvalues of L'prin}
h_1(\zeta)=\mu(|\xi'|^2+\zeta^2) \quad \text{(with multiplicity $d-1$)}, 
\quad
h_2(\zeta)=(\lambda+2\mu)(|\xi'|^2+\zeta^2) \quad \text{(with multiplicity $1$)}\,.
\end{equation}
We define the \emph{thresholds} of the continuous spectrum as the nonnegative real numbers $\Lambda_*$ such that the equation
\begin{equation*}
h_k(\zeta)=\Lambda_*
\end{equation*}
has a multiple real root for either $k=1$ or $k=2$. Formula \eqref{eigenvalues of L'prin} immediately implies that we have two thresholds for \eqref{one-dimensional spectral problem}:
\begin{equation}
\label{thresholds}
\Lambda_*^{(1)}= \mu |\xi'|^2,
\qquad 
\Lambda_*^{(2)}=(\lambda+2\mu) |\xi'|^2\,.
\end{equation}
Observe that, due to \eqref{conditions on lambda and mu}, we have $\Lambda_*^{(1)}<\Lambda_*^{(2)}$.

Formula \eqref{thresholds} then implies that the problem \eqref{one-dimensional spectral problem} has continuous spectrum $[\Lambda_*^{(1)}, +\infty)$ --- see, e.g., \cite[Appendix~A]{SaVa}. Furthermore, the thresholds partition the continuous spectrum into two zones $I^{(1)}:=(\Lambda_*^{(1)},\Lambda_*^{(2)})$ and $I^{(2)}:=(\Lambda_*^{(2)},+\infty)$, where the continuous spectrum has multiplicity $d-1$ and $d$, respectively.

\

{\bf Step 2:} {\it Eigenfunctions of the continuous spectrum}. Let $\bv_k(\zeta)$, $k=1,\dots, d-1$, be orthonormalised eigenvectors of $\LL'_\prin$ corresponding to the eigenvalue $h_1(\zeta)$, and let 
\begin{equation}
\label{eigenvector vd}
\bv_d(\zeta):=
\frac{1}{\sqrt{|\xi'|^2+\zeta^2}}
\begin{pmatrix}
\xi'
\\
\zeta
\end{pmatrix}
=\frac{1}{|\xi|}\xi
\end{equation}
be the normalised eigenvector of $\LL'_\prin$ corresponding to the eigenvalue $h_2(\zeta)$.
Let
\begin{equation*}
\label{zetas definition}
\zeta_1^\pm(\Lambda):=\pm \sqrt{\frac{\Lambda}{\mu}-|\xi'|^2} \quad \text{and}\quad \zeta_2^\pm(\Lambda):=\pm \sqrt{\frac{\Lambda}{\lambda+2\mu}- |\xi'|^2}
\end{equation*}
for $\Lambda\ge \mu|\xi'|^2$ and  $\Lambda\ge (\lambda+2\mu)|\xi'|^2$, respectively. Note that the quantities $\zeta_k^\pm(\Lambda)$ are solutions of $h_k(\zeta)-\Lambda=0$. Then, in view of elementary theory of matrix ordinary differential equations, we seek \emph{eigenfunctions of the continuous spectrum} (or \emph{generalised eigenfunctions}) for \eqref{one-dimensional spectral problem} in the form
\begin{multline}
\label{eigenfunctions in I1}
\bu(z;\Lambda)=\frac{1}{\sqrt{4\pi\mu}\left|\zeta_1^+(\Lambda)\right|^{1/2}}\sum_{j=1}^{d-1}\left(c_j^+\bv_j(\zeta_1^+(\Lambda))e^{i\zeta_1^+(\Lambda)z}
+
c_j^-\bv_j(\zeta_1^-(\Lambda))e^{i\zeta_1^-(\Lambda)z}
 \right)
 \\
 +C\, \bv_d\left(i\sqrt{|\xi'|^2-\frac{\Lambda}{\lambda+2\mu}} \right) e^{-\sqrt{|\xi'|^2-\frac{\Lambda}{\lambda+2\mu}}z}
\end{multline}
for $\Lambda \in I^{(1)}$, and in the form
\begin{multline}
\label{eigenfunctions in I2}
\bu(z;\Lambda)=\frac{1}{\sqrt{4\pi\mu}\left|\zeta_1^+(\Lambda)\right|^{1/2}}\sum_{j=1}^{d-1}\left(c_j^+\bv_j(\zeta_1^+(\Lambda))e^{i\zeta_1^+(\Lambda)z}
+
c_j^-\bv_j(\zeta_1^-(\Lambda))e^{i\zeta_1^-(\Lambda)z}
 \right)
 \\
 +
\frac{1}{\sqrt{4\pi(\lambda+2\mu)}\left|\zeta_2^+(\Lambda)\right|^{1/2}}
\left(c_d^+\bv_d(\zeta_2^+(\Lambda))e^{i\zeta_2^+(\Lambda)z}
+
c_d^-\bv_j(\zeta_2^-(\Lambda))e^{i\zeta_1^-(\Lambda)z}
 \right)
\end{multline}
for $\Lambda \in I^{(2)}$.
The complex numbers $c_j^\pm$ in \eqref{eigenfunctions in I1} and \eqref{eigenfunctions in I2} are called \emph{incoming} (-) and \emph{outgoing} (+) complex wave amplitudes, and are assumed not to be all zero.

\

{\bf Step 3:} {\it The scattering matrix}. By imposing that \eqref{eigenfunctions in I1} and \eqref{eigenfunctions in I2} satisfy the boundary conditions, one can express the coefficients $c_j^+$ in terms of the coefficients $c_j^-$. This defines the \emph{scattering matrices} $S^{(k)}(\Lambda)$, $k=1,2$, via the identities
\begin{equation*}
\begin{pmatrix}
c_1^+
\\
\vdots
\\
c_{d-1}^+
\end{pmatrix}
=
S^{(1)}(\Lambda)
\begin{pmatrix}
c_1^-
\\
\vdots
\\
c_{d-1}^-
\end{pmatrix}
\qquad \text{for} \ \Lambda\in I^{(1)}
\end{equation*}
and
\begin{equation*}
\begin{pmatrix}
c_1^+
\\
\vdots
\\
c_{d}^+
\end{pmatrix}
=
S^{(2)}(\Lambda)
\begin{pmatrix}
c_1^-
\\
\vdots
\\
c_{d}^-
\end{pmatrix}
\qquad \text{for} \ \Lambda\in I^{(2)}.
\end{equation*}
The matrix $S^{(1)}(\Lambda)$ (resp.~$S^{(2)}(\Lambda)$) is a $(d-1)\times (d-1)$ (resp.~$d\times d$) unitary matrix. The way in which the $c_j^\pm$ are arranged into a $(d-1)$-dimensional (resp.~$d$-dimensional) vector is unimportant and will not affect the quantities computed in the next steps.

\

{\bf Step 4:} {\it The phase shift}. Compute the \emph{phase shift}, defined as
\begin{equation}
\label{phase shift general}
\varphi_\aleph(\Lambda;\xi'):=
\begin{cases}
0 & \text{for} \ \Lambda \le \Lambda_1^*\\
\arg \det S^{k}(\Lambda)+ \mathfrak{s}^{(k)}  & \text{for} \ \Lambda \in I^{(k)}\\
\end{cases}
\end{equation}
where the branch of the multivalued function $\arg$ are chosen in such a way that $\varphi(\Lambda)$ is continuous in each interval $I^{(k)}$, and the shifts $\mathfrak{s}^{(k)}$ are constants determined by the requirement that the jump of the phase shift at the thresholds satisfy
\begin{equation}
\label{equation jumps}
\frac{1}{\pi}\lim_{\epsilon\to 0^+}\left(\varphi(\Lambda_*^{(k)}+\epsilon)-\varphi(\Lambda_*^{(k)}-\epsilon)\right)= j_*^{(k)}-\frac{m_k}{2}, \qquad k\in\{1,2\}\,.
\end{equation}
Here $m_k$ is the multiplicity of the eigenvalue $h_k$ and $j_*^{(k)}$ is the number of linearly independent vectors $\bv$ such that
\begin{equation}
\label{solution for thresholds}
\bv e^{i \zeta_+(\Lambda_*^{(k)})z}+\mathbf{f}(z)
\end{equation}
is a solution of the one-dimensional problem \eqref{one-dimensional spectral problem}, with $\mathbf{f}(z)=o(1)$ as $z\to +\infty$. The threshold $\Lambda_*^{(k)}$ is called \emph{rigid} if $j_*^{(k)}=0$ and \emph{soft} if $j_*^{(k)}=m_k$

\

{\bf Step 5:} {\it The one-dimensional counting function}. Compute the \emph{one-dimensional counting function}, defined as
\begin{equation}
\label{one dimensional counting function}
N_{\aleph,\mathrm{1D}}(\Lambda;\xi'):=\#\{\text{eigenvalues of \eqref{one-dimensional spectral problem} strictly smaller than}\ \Lambda\}\,.
\end{equation}

\

{\bf Step 6:} {\it The spectral shift function}. Compute the \emph{spectral shift function}, defined as
\begin{equation}
\label{spectral shift function}
\mathrm{shift}_\aleph(\Lambda;\xi'):=\frac{1}{2\pi}\varphi_\aleph(\Lambda;\xi')+N_{\aleph,\mathrm{1D}}(\Lambda;\xi').
\end{equation}

\

Then we have the following.

\begin{thm}
\label{theorem algorithm second Weyl coefficient}
The second Weyl coefficient\footnote{Recall that the second Weyl coefficient and the second coefficient in the asymptotic expansion for the eigenvalue counting function~\eqref{theorem two-term asymptotics with billiards conditions equation 1} are related in accordance with~\eqref{main theorem equation 0} --- see also~\eqref{relation between C and c}.} is given by
\begin{equation}
\label{theorem algorithm second Weyl coefficient equation1}
c_{d-2}^\aleph=\frac{d-1}{2(2\pi)^{d-1}}\int_{T^*\partial M} \mathrm{shift}_\aleph(1;\xi')\, \dr x'\,\dr\xi' =\frac{d-1}{2}\frac{\mathrm{Vol}_{d-1}(\partial M)}{(2\pi)^{d-1}}\int_{\mathbb{R}^{d-1}} \mathrm{shift}_\aleph(1;\xi')\,\dr\xi'\,.
\end{equation}
\end{thm}

Theorem~\ref{theorem algorithm second Weyl coefficient} is a specialisation to the case at hand of \cite[Theorem~2]{Vas86} (once the latter has been extended to systems).

\subsection{Invariant subspaces}
\label{Invariant subspaces}

Implementing the algorithm from subsection~\ref{A streamlined algorithm} as written for an arbitrary dimension $d$ is quite tricky. Rather than applying our algorithm to~\eqref{L u=Lambda u}, \eqref{DF bc} and~\eqref{L u=Lambda u}, \eqref{FD bc} directly, we shall first simplify the problem by decomposing the general $d$-dimensional problem into a two-dimensional analogue plus a much simpler $(d-2)$-dimensional problem. The arguments in this subsection can be traced back, in essence, to observations by Dupuis--Mazo--Onsager \cite{Onsager}, see also \cite[Section~3]{2terms}. The key idea is to decompose elastic waves into two polarised components: one polarised in the plane of propagation and the other normally to it.

\

To this end, suppose we have fixed $\xi'\in \mathbb{R}^{d-1}$, $\xi'\ne 0$, and define
\begin{equation*}
P:=\mathrm{span}\left\{\frac{1}{|\xi'|}\begin{pmatrix}
\xi'
\\
0
\end{pmatrix},
\begin{pmatrix}
0'
\\
1
\end{pmatrix}
\right\}\subset \mathbb{R}^d\,.
\end{equation*}
Let $P^\perp$ be the orthogonal complement of $P$ in $\mathbb{R}^d$. One can easily check the following facts.
\begin{fact}
\label{facts invariant subspaces}
\begin{enumerate}[(i)]
\item
For every $\zeta \in \mathbb{R}$ the eigenvector $\bv_d(\zeta)$ \eqref{eigenvector vd} is an element of $P$.

\item
For every $\zeta \in \mathbb{R}$ the orthogonal subspaces $P$ and $P^\perp$ are invariant subspaces of $\LL_\prin(\xi',\zeta)=\LL'_\prin(\zeta)$ --- recall~\eqref{Lprin}.

\item
The restriction $\left.\LL'_\prin\right|_{P^\perp}$ of $\LL'_\prin$ to $P^\perp$ has one eigenvalue, $\mu|\xi|^2$, of multiplicity $d-2$.

\item
The restriction $\left.\LL'_\prin\right|_{P}$ of $\LL'_\prin$ to $P$ has two eigenvalues, $\mu|\xi|^2$ and $(\lambda+2\mu)|\xi|^2$, each of multiplicity $1$.
\end{enumerate}
\end{fact}

The decomposition $\mathbb{R}^{d}=P\oplus P^\perp$ induces a corresponding decomposition at the level of vector fields. Let 
\begin{equation}
\label{space bold P}
\mathbf{P}:=\{\bu \in C^\infty[0,+\infty)\ |\ \bu(z)\in P \quad \forall z\in[0,+\infty)\}\,,
\end{equation}
\begin{equation}
\label{space bold P perp}
\mathbf{P}^\perp:=\{\bu \in C^\infty[0,+\infty)\ |\ \bu(z)\in P^\perp \quad \forall z\in[0,+\infty)\}\,.
\end{equation}

\begin{prop}
\label{proposition invariant subspaces}
The vector spaces \eqref{space bold P} and \eqref{space bold P perp} are invariant subspaces for the operator \eqref{L'}, compatible with mixed boundary conditions $\mathrm{DF}$ and $\mathrm{FD}$. Namely,
\begin{equation}
\label{proposition invariant subspaces equation 1}
\LL'\mathbf{P}\subset \mathbf{P}, \qquad \LL'\mathbf{P}^\perp\subset \mathbf{P}^\perp,
\end{equation}
and
\begin{equation}
\label{proposition invariant subspaces equation 2}
\left.\mathcal{B}_\aleph' \mathbf{P}\right|_{z=0}\subset P, \qquad \left.\mathcal{B}_\aleph' \mathbf{P}\right|_{z=0}\subset P^\perp, \qquad \aleph\in\{\mathrm{DF},\mathrm{FD}\}\,.
\end{equation}
\end{prop}
\begin{proof}
A generic element of $\mathbf{P}$ reads
\begin{equation}
\label{proof proposition invariant subspaces equation 1}
\mathbf{u}_{\parallel}=\frac{1}{\|\xi'\|}\begin{pmatrix}
\xi'
\\
0
\end{pmatrix}
f_1(z)
+
\begin{pmatrix}
0'
\\
1
\end{pmatrix} 
f_2(z)\,, \qquad f_1,f_2 \in C^\infty[0,+\infty)
\end{equation}
whereas a generic element of $\mathbf{P}^\perp$ reads
\begin{equation}
\label{proof proposition invariant subspaces equation 2}
\mathbf{u}_\perp(z)=\sum_{j=1}^{d-2}\begin{pmatrix}
\psi_j
\\
0
\end{pmatrix}
f_j(z), \qquad f_j\in C^\infty[0,+\infty),
\end{equation}
where the $\psi_j$'s, $j=1,\dots,d-2$,  are linearly independent columns in $\mathbb{R}^{d-1}$ orthogonal to $\xi'$. 

Formula~\eqref{proposition invariant subspaces equation 1} follows from \cite[Lemma~3.1(a)]{2terms}.

Substituting~\eqref{proof proposition invariant subspaces equation 1} and~\eqref{proof proposition invariant subspaces equation 2} into \eqref{B'} we obtain
\begin{equation}
\label{10 August 2023 equation 5}
\left.\left(\mathcal{B}'_{DF}\mathbf{u}_\parallel\right)\right|_{z=0}=
\frac{1}{\|\xi'\|}
\begin{pmatrix}
\xi'
\\
0
\end{pmatrix}
f_1(0)
-(\lambda+2\mu)
\begin{pmatrix}
0'
\\
1
\end{pmatrix} 
\frac{\dr f_2}{\dr z}(0)\,,
\end{equation}
\begin{equation}
\label{10 August 2023 equation 6}
\left.\left(\mathcal{B}'_{DF}\mathbf{u}_\perp\right)\right|_{z=0}=
\sum_{j=1}^{d-2}\begin{pmatrix}
\psi_j
\\
0
\end{pmatrix}
\,f_j(0)
\end{equation}
for $\aleph=\mathrm{DF}$ and
\begin{equation}
\label{10 August 2023 equation 3}
\left.\left(\mathcal{B}'_{FD}\mathbf{u}_\parallel\right)\right|_{z=0}=
-\mu\frac{1}{\|\xi'\|}
\begin{pmatrix}
\xi'
\\
0
\end{pmatrix}
\frac{\dr f_1}{\dr z}(0)
+
\begin{pmatrix}
0'
\\
1
\end{pmatrix} 
f_2(0)\,,
\end{equation}
\begin{equation}
\label{10 August 2023 equation 4}
\left.\left(\mathcal{B}'_{FD}\mathbf{u}_\perp\right)\right|_{z=0}=
-\mu\sum_{j=1}^{d-2}\begin{pmatrix}
\psi_j
\\
0
\end{pmatrix}
\frac{\dr f_j}{\dr z}(0)
\end{equation}
for $\aleph=\mathrm{FD}$. Formulae~\eqref{10 August 2023 equation 5}--\eqref{10 August 2023 equation 4} imply that mixed boundary conditions preserve our invariant subspaces, so that~\eqref{proposition invariant subspaces equation 2} holds. This concludes the proof.
\end{proof}

\begin{rem}
The crucial property established by Proposition~\ref{proposition invariant subspaces} is expressed by formula~\eqref{proposition invariant subspaces equation 2}.
An analogue of Proposition~\ref{proposition invariant subspaces} for ``pure" Dirichlet and free boundary conditions was proved in \cite{2terms}. That this extends to a mixture of the two is not clear \emph{a priori}, because both the operator and boundary conditions mix up components in a nontrivial fashion. Indeed, if one, say, imposes different boundary conditions in different directions \emph{along} the boundary, the statement of the proposition is false.
\end{rem}

Let $\Pi_P$ be the orthogonal projection in $\mathbb{R}^d$ onto $P$, and let us define
\begin{equation}
\LL'_{P,\aleph}:=\left.\LL'_\aleph\right|_{\Pi_PD(\LL'_\aleph)} \qquad\text{and}\qquad \LL'_{\perp,\aleph}:=\left.\LL'_\aleph\right|_{(I-\Pi_{P})D(\LL'_\aleph)} 
\end{equation}
to be the restriction of the operator $\LL'$ with boundary conditions $\aleph\in\{\mathrm{DF}, \mathrm{FD}\}$ to the invariant subspaces of its domain $D(\LL'_\aleph)$ induced by \eqref{space bold P} and \eqref{space bold P perp} by combining Proposition~\ref{proposition invariant subspaces} with a standard density argument. It then follows that operator $\LL'_\aleph$ decomposes as
\begin{equation}
\label{decompositon of L'}
\LL'_\aleph=\LL'_{P,\aleph}\oplus\LL'_{\perp,\aleph},
\end{equation}
so that, by the Spectral Theorm, we have
\begin{equation}
\label{decompositon of shift}
\mathrm{shift}_\aleph=\mathrm{shift}_{P,\aleph}+\mathrm{shift}_{\perp,\aleph}\,.
\end{equation}
In other words, the spectral shift function for the problem \eqref{one-dimensional spectral problem} can be obtained by computing the spectral shift functions for $\LL'_{P,\aleph}$ and $\LL'_{\perp,\aleph}$ separately, and adding up the results in the end.

By examining the structure of our equations, it is not hard to see that $\mathrm{shift}_{P,\aleph}$ coincides with the the spectral shift function for the problem \eqref{one-dimensional spectral problem} in the special case $d=2$ (we will revisit this point more formally in subsection~\ref{The proof}). Therefore, in view of Theorem~\ref{theorem algorithm second Weyl coefficient} and formula~\eqref{decompositon of shift}, the decomposition~\eqref{decompositon of L'} reduces the problem at hand to computing 
\begin{enumerate}[(i)]
\item the spectral shift function for \eqref{one-dimensional spectral problem} in two dimensions and
\item the spectral shift function of the restriction of our operator to normally polarised vector fields in arbitrary dimension $d>2$.
\end{enumerate}

\subsection{The proof}
\label{The proof}

We are now ready to prove Theorem~\ref{main theorem}.

\

Due to rotational symmetry, we observe that the spectral shift function will only depend on $\xi'$ via its norm $|\xi'|$. Therefore, it suffices to implement our algorithm and determine the spectral shift function in the special case
\begin{equation}
\label{special xi'}
\overline{\xi'}=\begin{pmatrix}
0
\\
\vdots
\\
0
\\
1
\end{pmatrix}\in \mathbb{R}^{d-1}\,.
\end{equation}
The general case can then be recovered by rescaling the spectral parameter as
\begin{equation}
\label{rescaling the spectral parameter}
\Lambda \mapsto \frac{\Lambda}{|\xi'|^2}
\end{equation}
at the very end.

\

In the next two subsections we will assume that $\xi'$ has been chosen in accordance with \eqref{special xi'}.

\subsubsection{Computing $\mathrm{shift}_{P,\aleph}$: the two-dimensional case}

On account of \eqref{special xi'}, the domain of $\LL'_{P,\aleph}$ is comprised of vector functions of the form
\begin{equation}
\label{vector function essentially 2d}
\begin{pmatrix}
0
\\
\vdots
\\
0
\\
f_1(z)
\\
f_2(z)
\end{pmatrix}\,.
\end{equation}
Furthermore, $\LL'_{P,\aleph}$ acts on \eqref{vector function essentially 2d} as the one-dimensional operator associated with the full elasticity operator in two spatial dimensions. More precisely, let $\LL'_{2,\aleph}$ be the one-dimensional operator \eqref{one-dimensional spectral problem} associated with the operator \eqref{L} for $d=2$ and boundary conditions $\mathcal{B}_\aleph$. Then we have
\begin{equation*}
\LL'_{P,\aleph}\begin{pmatrix}
0
\\
\vdots
\\
0
\\
f_1(z)
\\
f_2(z)
\end{pmatrix}
=
\begin{pmatrix}
0
\\
\vdots
\\
0
\\
\LL'_{2,\aleph}
\begin{pmatrix}
f_1(z)
\\
f_2(z)
\end{pmatrix}
\end{pmatrix}.
\end{equation*}
See also Fact~\ref{facts invariant subspaces}(iv). This implies
\begin{equation}
\label{spectral shifr parallel equal spectral shift 2D}
\mathrm{shift}_{P,\aleph}=\mathrm{shift}_{2,\aleph}.
\end{equation}

The goal of this subsection is then to prove the following.
\begin{prop}
\label{prop spectral shift 2d}
We have\footnote{Observe that in two dimensions formula~\eqref{special xi'} reads $\overline{\xi'}=1$.}
\begin{equation}
\label{prop spectral shift 2d equation 1}
\mathrm{shift}_{2,\aleph}(\Lambda;1)= @\frac14\mathbb{1}_{(\mu, \lambda+2\mu)}(\Lambda) \quad \text{where}\quad 
@
=
\begin{cases}
+ &\text{for }\aleph=\mathrm{DF}\\
- &\text{for }\aleph=\mathrm{FD}
\end{cases}
\end{equation}
and $\mathbb{1}_A$ denotes the characteristic function of the set $A$.
\end{prop}

In order to prove Proposition~\ref{prop spectral shift 2d} let us implement the algorithm from subsection~\ref{A streamlined algorithm}. The associated one-dimensional spectral problem \eqref{one-dimensional spectral problem} has continuous spectrum $[\mu, +\infty)$, with thresholds
\begin{equation*}
\label{thresholds 2D}
\Lambda_*^{(1)}= \mu\,,
\qquad 
\Lambda_*^{(2)}=(\lambda+2\mu)\,.
\end{equation*}
The latter partition the continuous spectrum into two intervals $I^{(1)}=(\mu, \lambda+2\mu)$ and $I^{(2)}=(\lambda+2\mu, +\infty)$ of multiplicity $1$ and $2$, respectively.

The eigenfunctions of the continuous spectrum read
\begin{multline}
\label{25 July 2023 equation 5}
\mathbf{u}(z;\Lambda)
=
\frac{1}{\sqrt{4\pi} \sqrt{\Lambda}}
\begin{pmatrix}
\left( \frac{\Lambda}{\mu}-1\right)^{1/4}
\left[c_1^- e^{-i \left( \frac{\Lambda}{\mu}-1\right)^{1/2} z} - c_1^+ e^{i \left( \frac{\Lambda}{\mu}-1\right)^{1/2} z}\right]
\\
\left( \frac{\Lambda}{\mu}-1\right)^{-1/4}
\left[c_1^- e^{-i \left( \frac{\Lambda}{\mu}-1\right)^{1/2} z} + c_1^+ e^{i \left( \frac{\Lambda}{\mu}-1\right)^{1/2} z}\right]
\end{pmatrix}
\\
+
C \,
\sqrt{\frac{\lambda+2\mu}{\Lambda}}
\begin{pmatrix}
1
\\
i \left( 1-\frac{\Lambda}{\lambda+2\mu}\right)^{1/2}
\end{pmatrix}
e^{-\sqrt{1-\frac{\Lambda}{\lambda+2\mu}}\, z}
\end{multline}
for $\Lambda\in I^{(1)}$ and 
\begin{multline}
\label{26 July 2023 equation 1}
\mathbf{u}(z;\Lambda)
=
\frac{1}{\sqrt{4\pi} \sqrt{\Lambda}}
\begin{pmatrix}
\left( \frac{\Lambda}{\mu}-1\right)^{1/4}
\left[c_1^- e^{-i \left( \frac{\Lambda}{\mu}-1\right)^{1/2} z} - c_1^+ e^{i  \left( \frac{\Lambda}{\mu}-1\right)^{1/2} z}\right]
\\
\left( \frac{\Lambda}{\mu}-1\right)^{-1/4}
\left[c_1^- e^{-i  \left( \frac{\Lambda}{\mu}-1\right)^{1/2} z} + c_1^+ e^{i \left( \frac{\Lambda}{\mu}-1\right)^{1/2} z}\right]
\end{pmatrix}
\\
+
\frac{1}{\sqrt{4\pi} \sqrt{\Lambda}}
\begin{pmatrix}
\left( \frac{\Lambda}{\lambda+2\mu}-1\right)^{-1/4}
\left[c_2^- e^{-i \left( \frac{\Lambda}{\lambda+2\mu}-1\right)^{1/2} z} + c_2^+ e^{i \left( \frac{\Lambda}{\lambda+2\mu}-1\right)^{1/2} z}\right]
\\
\left( \frac{\Lambda}{\lambda+2\mu}-1\right)^{1/4}
\left[-c_2^- e^{-i \left( \frac{\Lambda}{\lambda+2\mu}-1\right)^{1/2} z} + c_2^+ e^{i \left( \frac{\Lambda}{\lambda+2\mu}-1\right)^{1/2} z}\right]
\end{pmatrix}
\end{multline}
for $\Lambda\in I^{(2)}$.

By imposing that~\eqref{25 July 2023 equation 5} and~\eqref{26 July 2023 equation 1} satisfy mixed boundary conditions \eqref{one-dimensional spectral problem}, \eqref{B'} we obtain the scattering matrices
\begin{equation}
\label{scattering 2d DF}
S_{\mathrm{DF}}(\Lambda)=
\begin{cases}
1 & \text{for} \quad \Lambda\in (\mu, \lambda+2\mu)\,,
\\
\begin{pmatrix}
-1 & 0\\
0 & 1
\end{pmatrix} 
&\text{for} \quad \Lambda\in (\lambda+2\mu, +\infty)
\end{cases}
\end{equation}
and
\begin{equation}
\label{scattering 2d FD}
S_{\mathrm{FD}}(\Lambda)=
\begin{cases}
-1 & \text{for} \quad \Lambda\in (\mu, \lambda+2\mu)\,,
\\
\begin{pmatrix}
-1 & 0\\
0 & 1
\end{pmatrix} 
&\text{for} \quad \Lambda\in (\lambda+2\mu, +\infty)\,.
\end{cases}
\end{equation}

\begin{lem}
\label{lemma thresholds 2D}
We have the following:
\begin{equation}
\label{lemma thresholds 2D equation 1}
\text{the threshold $\Lambda_*^{(1)}$ is}\quad
\begin{cases}
\text{soft} & \text{for} \quad \aleph=\mathrm{DF}\\
\text{rigid} & \text{for} \quad \aleph=\mathrm{FD}
\end{cases}\ ,
\end{equation}
\begin{equation}
\label{lemma thresholds 2D equation 2}
\text{the threshold $\Lambda_*^{(2)}$ is}\quad
\begin{cases}
\text{rigid} & \text{for} \quad \aleph=\mathrm{DF}\\
\text{soft} & \text{for} \quad \aleph=\mathrm{FD}
\end{cases}\ .
\end{equation}
\end{lem}
\begin{proof}
In accordance with \eqref{solution for thresholds}, for $\Lambda=\Lambda_*^{(1)}$ we seek solutions in the form
\begin{equation}
\label{27 July 2023 equation 3}
c_1 \begin{pmatrix}
0
\\
1
\end{pmatrix}
+
c_2
\begin{pmatrix}
1
\\
i\sqrt{\frac{\lambda+\mu}{\lambda+2\mu}}
\end{pmatrix}
e^{- \sqrt{\frac{\lambda+\mu}{\lambda+2\mu}}z}\,
\end{equation}
for some constants $c_1$ and $c_2$. Substituting \eqref{27 July 2023 equation 3} into our boundary conditions one finds that when $\aleph=\mathrm{DF}$ the function \eqref{27 July 2023 equation 3} satisfies boundary conditions for any $c_1\in \mathbb{R}$ and $c_2=0$, whereas when $\aleph=\mathrm{FD}$ the function \eqref{27 July 2023 equation 3} only satisfies boundary conditions for $c_1=c_2=0$. This gives us \eqref{lemma thresholds 2D equation 1}.

For $\Lambda=\Lambda_*^{(2)}$ we seek solutions in the form
\begin{equation}
\label{27 July 2023 equation 7}
c
\begin{pmatrix}
1
\\
0
\end{pmatrix}
\end{equation}
for some constant $c$. Now, it is easy to see that \eqref{27 July 2023 equation 7} satisfies $\mathrm{FD}$ boundary conditions for any $c\in \mathbb{R}$, whereas it satisfies $\mathrm{DF}$ boundary conditions only for $c=0$. This gives us \eqref{lemma thresholds 2D equation 2} and completes the proof.
\end{proof}

\begin{lem}
\label{lemma no eigenvalues 2d}
The operator $\LL'_{2,\aleph}$ does not have eigenvalues below or embedded into the continuous spectrum for either set of mixed boundary conditions $\aleph=\mathrm{DF},\mathrm{FD}$.
\end{lem}
\begin{proof}
For $\Lambda\in (0,\mu)$ we seek an eigenfunction in the form
\begin{equation}
\label{31 July 2023 equation 1}
c_1
\begin{pmatrix}
-i\left( 1-\frac{\Lambda}{\mu}\right)^{1/2}
\\
1
\end{pmatrix}
e^{- \sqrt{1-\frac{\Lambda}{\mu}} z}
\\
+
c_2 \,
\begin{pmatrix}
1
\\
i \left( 1-\frac{\Lambda}{\lambda+2\mu}\right)^{1/2}
\end{pmatrix}
e^{-\sqrt{1-\frac{\Lambda}{\lambda+2\mu}}\, z}\,.
\end{equation}
Substituting \eqref{31 July 2023 equation 1} into the $\mathrm{FD}$ boundary conditions we obtain
\begin{equation*}
\label{31 July 2023 equation 2}
\begin{pmatrix}
i \left( 1-\frac{\Lambda}{\mu}\right) & - \left( 1-\frac{\Lambda}{\lambda+2\mu}\right)^{1/2}
\\
1 & i \left( 1-\frac{\Lambda}{\lambda+2\mu}\right)^{1/2}
\end{pmatrix}
\begin{pmatrix}
c_1
\\
c_2
\end{pmatrix}
=
\begin{pmatrix}
0
\\
0
\end{pmatrix}\,.
\end{equation*}
The latter has a nontrivial solution if and only if
\begin{equation}
\label{31 July 2023 equation 3}
\chi(\Lambda)=\left( 1-\frac{\Lambda}{\lambda+2\mu}\right)^{1/2} \left[-\left( 1-\frac{\Lambda}{\mu}\right)+1 \right]
=\left( 1-\frac{\Lambda}{\lambda+2\mu}\right)^{1/2} \frac{\Lambda}{\mu} =0\,.
\end{equation}
But the characteristic equation \eqref{31 July 2023 equation 3} does not admit solutions in $(0,\mu)$. The case of $\mathrm{DF}$ is analogous, with no eigenfunctions in $(0,\mu)$, and we omit the details. All in all, there are no solutions below bottom of the essential spectrum for either set of mixed boundary conditions.

The threshold $\Lambda=\mu$ is not an eigenvalue. Indeed, an eigenfunction of the form
\begin{equation*}
\label{31 July 2023 equation 4}
c
\begin{pmatrix}
1
\\
i\sqrt{\frac{\lambda+\mu}{\lambda+2\mu}}
\end{pmatrix}
e^{- \sqrt{\frac{\lambda+\mu}{\lambda+2\mu}}z}
\end{equation*}
satisfies boundary conditions~\eqref{B'} only if $c=0$.

For $\Lambda\in (\mu,\lambda+2\mu)$ we seek an eigenfunction in the form
\begin{equation*}
\label{31 July 2023 equation 5}
c \,
\begin{pmatrix}
1
\\
i \left( 1-\frac{\Lambda}{\lambda+2\mu}\right)^{1/2}
\end{pmatrix}
e^{-\sqrt{1-\frac{\Lambda}{\lambda+2\mu}}\, z}\,.
\end{equation*}
Once again, the latter satisfies either set of mixed boundary conditions~\eqref{B'} only if $c=0$. Therefore, there are no eigenvalues in $(\mu,\lambda+2\mu)$.

Finally, it is easy to see that $\Lambda=\lambda+2\mu$ is not an eigenvalue, and that there are no square integrable solutions of our one-dimensional spectral problem for values of the spectral parameter $\Lambda>\lambda+2\mu$.
\end{proof}

Now, Lemma~\ref{lemma no eigenvalues 2d} implies that the one-dimensional counting function vanishes identically. Therefore, on account of \eqref{spectral shift function} and \eqref{phase shift general}--\eqref{equation jumps}, combining \eqref{scattering 2d DF}, \eqref{scattering 2d FD} with Lemma~\ref{lemma thresholds 2D} one arrives at~\eqref{prop spectral shift 2d equation 1}.

\subsubsection{Computing $\mathrm{shift}_{\perp,\aleph}$: normally polarised waves}

Let us now examine our one-dimensional spectral problem restricted to the subspace $\mathbf{P}^\perp$ \eqref{space bold P perp}. The goal of this subsection is to prove the following.
\begin{prop}
\label{prop spectral shift perp}
We have
\begin{equation}
\label{prop spectral shift perp equation 1}
\mathrm{shift}_{\perp,\aleph}(\Lambda;\overline{\xi'})= @\frac{d-2}{4}\mathbb{1}_{(\mu, +\infty)}(\Lambda) \quad \text{where}\quad 
@
=
\begin{cases}
- &\text{for }\aleph=\mathrm{DF}\\
+ &\text{for }\aleph=\mathrm{FD}
\end{cases}.
\end{equation}
\end{prop}

In order to prove Proposition~\ref{prop spectral shift perp} let us implement the algorithm from subsection~\ref{A streamlined algorithm}.

One can easliy check that when restricted to normally polarised waves the operator $\LL'$ acts as
\begin{equation*}
\label{L' perp}
\LL'_{\perp,\aleph}\bu=\mu\left(1-\frac{\dr^2}{\dr z^2}\right)\bu.
\end{equation*}
This implies that the one-dimensional spectral problem \eqref{one-dimensional spectral problem} for $\LL'=\LL'_{\perp,\aleph}$ has only one threshold
\begin{equation}
\label{threshold d-2}
\Lambda_*=\mu,
\end{equation}
and the essential spectrum $[\mu, +\infty)$ has multiplicity $d-2$.

For $\Lambda>\mu$ the eigenfunctions of the continuous spectrum read
\begin{equation}
\label{9.1 - Second subspace eigenfunction}
\mathbf{u}(z;\Lambda) = \sum_{j=1}^{d-2} \mathbf{e}_j \left(c_j^+ e^{i\sqrt{\frac{\Lambda}{\mu}-1}} + c_j^- e^{-i\sqrt{\frac{\Lambda}{\mu}-1}} \right)\,,
\end{equation}
where $(\mathbf{e}_j)_\alpha=\delta_{j\alpha}$.

By imposing that~\eqref{9.1 - Second subspace eigenfunction} satisfy mixed boundary conditions \eqref{B'} we obtain the scattering matrices
\begin{equation}
\label{scattering matrix d-2}
S_\aleph(\Lambda)=@ I_{d-2}\,, \qquad \quad
@=\begin{cases}
- & \text{for}\ \,\aleph=\mathrm{DF}\\
+ & \text{for}\ \,\aleph=\mathrm{FD}
\end{cases}, \quad \Lambda\in (\mu, +\infty).
\end{equation}
where $I_{d-2}$ is the $(d-2)$-dimensional identity matrix. 

\begin{lem}
\label{lemma thresholds d-2}
The threshold \eqref{threshold d-2} is
\begin{equation}
\label{lemma thresholds d-2 equation 1}
\begin{cases}
\text{rigid} & \text{for} \quad \aleph=\mathrm{DF}\\
\text{soft} & \text{for} \quad \aleph=\mathrm{FD}
\end{cases}\ .
\end{equation}
\end{lem}
\begin{proof}
In accordance with \eqref{solution for thresholds}, for $\Lambda=\Lambda_*$ we seek solutions in the form
\begin{equation}
\label{proof lemma thresholds d-2 equation 1}
\begin{pmatrix}
c_1 \\ c_2 \\ ... \\ c_{d-2} \\ 0 \\ 0
\end{pmatrix}
\end{equation}
for some constants $c_j$, $j=1,\ldots, d-2$. Substituting the latter into \eqref{B'} one immediately sees that~\eqref{proof lemma thresholds d-2 equation 1} satisfies $\mathrm{DF}$ boundary conditions only if all the $c_j$'s vanish, whereas it satisfies $\mathrm{FD}$ boundary conditions for any choice of constants $c_j$'s. Hence, one has \eqref{lemma thresholds d-2 equation 1}.
\end{proof}

\begin{lem}
\label{lemma no eigenvalues perp}
The operator $\LL'_{\perp,\aleph}$, $\aleph\in\{\mathrm{DF}, \mathrm{FD}\}$, does not have eigenvalues, either below or embedded into the continuous spectrum.
\end{lem}
\begin{proof}
For $\Lambda\in (0,\mu)$ we seek an eigenfunction in the form
\begin{equation*}
\begin{pmatrix}
c_1 \\ c_2 \\ ... \\ c_{d-2} \\ 0 \\ 0
\end{pmatrix} e^{-\sqrt{1-\frac{\Lambda}{\mu}}}\,.
\end{equation*}
But the latter does not satisfy either set of mixed boundary conditions unless all the constants $c_j$ vanish; therefore, there are no eigenvalues in $(0,\mu)$.

Finally, it is easy to see that there are no square integrable solutions of the one-dimensional spectral problem for values of the spectral parameter $\Lambda\ge \mu$. Hence, there are no eigenvalues in $[\mu,+\infty)$ either.
\end{proof}

As in the previous subsection, Lemma~\ref{lemma no eigenvalues perp} implies that the one-dimensional counting function vanishes identically. Therefore, on account of \eqref{spectral shift function} and \eqref{phase shift general}--\eqref{equation jumps}, combining \eqref{scattering matrix d-2} with Lemma~\ref{lemma thresholds d-2} one arrives at~\eqref{prop spectral shift perp equation 1}.

\subsubsection{Putting things together}
\label{Putting things together}

Combining Proposition~\ref{prop spectral shift 2d}, Proposition~\ref{prop spectral shift perp}, and formulae~\eqref{rescaling the spectral parameter}, \eqref{decompositon of shift}, \eqref{spectral shifr parallel equal spectral shift 2D} we obtain
\begin{equation}
\label{spectral shift general}
\mathrm{shift}_\aleph(\Lambda;\xi')
=
@\begin{cases}
0 & \text{for }\ \Lambda<\mu \,|\xi'|^2
\\
\frac{d-3}{4} & \text{for }\ \mu \,|\xi'|^2<\Lambda<(\lambda+2\mu)\, |\xi'|^2
\\
\frac{d-2}{4} & \text{for }\ \Lambda>(\lambda+2\mu)\, |\xi'|^2
\end{cases}
\quad \text{with}\quad
@=\begin{cases}
- & \text{for}\ \,\aleph=\mathrm{DF},\\
+ & \text{for}\ \,\aleph=\mathrm{FD}.
\end{cases}
\end{equation}
Substituting \eqref{spectral shift general} into \eqref{theorem algorithm second Weyl coefficient equation1} and integrating we arrive at~\eqref{main theorem equation 1}. This completes the proof of Theorem~\ref{main theorem}.

\section{Explicit examples}
\label{Explicit examples}

In this section we verify our formulae for the second Weyl coefficients by examining the asymptotics of the eigenvalue counting function for explicit examples: the disk, and flat cylinders in dimensions $d=2$ and $d=3$.

\

The choice of examples is motivated by the fact that they possess the following properties.
\begin{enumerate}[(i)]
\item 
They allow for separation of variables for the operator of linear elasticity with mixed boundary conditions.

\item
They satisfy the conditions on branching Hamiltonian billiards from Theorem~\ref{theorem two-term asymptotics with billiards conditions}, so that the two-term asymptotics for the counting function is valid.

\item For flat cylinders, variables separate completely and one can write down the full spectrum explicitly. Therefore, unlike in \cite{2terms}, we can verify our formulae \emph{analytically}, using asymptotic expansions for certain number-theoretic series determined by our eigenvalues.
\end{enumerate}

\subsection{Two-dimensional examples}
\label{Two-dimensional examples}

\subsubsection{The disk}
\label{The disk}

Let $M\subset \mathbb{R}^2$ be the unit disk and let us work in standard polar coordinates $(r,\theta)$. Following \cite[Chapter~XIII]{MoFe} (see also~\cite{LMS}), we introduce a fictitious third coordinate $z$ orthogonal to the disk and seek solutions in the form
\begin{equation}
\label{solutions disk}
\bu(r,\theta)=\grad \psi_1(r,\theta)+ \curl\left(\psi_2(r,\theta)\,\hat{\mathbf{z}}\right)\,,
\end{equation}
where $\hat{\mathbf{z}}$ is the unit vector in the direction of $z$ and $\psi_j$, $j=1,2$, are auxiliary scalar potentials. Substituting \eqref{solutions disk} into \eqref{L u=Lambda u} one obtains that the scalar potentials must satisfy the Helmholtz equations
\begin{equation}
\label{helmholz 2d disk}
-\Delta \psi_j=\omega_{j,\Lambda} \psi_j, \qquad j=1,2,
\end{equation}
\begin{equation}
\label{omegas helmholtz}
\omega_{1,\Lambda}:= \frac{\Lambda}{\lambda+2\mu}, \qquad \omega_{2,\Lambda}:= \frac{\Lambda}{\mu}\,.
\end{equation}
But now the general solution to \eqref{helmholz 2d disk} regular at $r=0$ reads
\begin{equation}
\label{solution helmholtz 2d}
\psi_j(r,\phi)=c_{j,0} J_0\left(\sqrt{\omega_{j,\Lambda}} r\right)+\sum_{k=1}^\infty J_k\left(\sqrt{\omega_{j,\Lambda}} r\right) \left(c_{j,k,+} \mathrm{e}^{\mathrm{i}k\theta}+c_{j,k,-} \mathrm{e}^{-\mathrm{i}k\theta}\right),
\end{equation}
where the $J_k$'s are Bessel functions of the first kind. By substituting \eqref{solution helmholtz 2d} and imposing that \eqref{solutions disk} satisfies $\mathrm{DF}$ boundary conditions
\[
\left.\begin{pmatrix}
(\lambda +2 \mu ) \,\partial_r u_1+\lambda  (u_1+   \partial_\theta u_2)
   \\
   u_2
\end{pmatrix}\right|_{r=1} =
   \begin{pmatrix}
   0
   \\
   0
\end{pmatrix}\,,
\]
one obtains the secular equation
\begin{multline}
\label{secular equation DN}
\mu k J_k\left(\sqrt{\omega_{2,\Lambda}}\right) \left[\omega_{2,\Lambda}
   J_k\left(\sqrt{\omega_{1,\Lambda}}\right)-2  \sqrt{\omega_{1,\Lambda}} J_{k+1}\left(\sqrt{\omega_{1,\Lambda}}\right)\right]
\\   
   +\mu \sqrt{\omega_{2,\Lambda}} J_{k+1}\left(\sqrt{\omega_{2,\Lambda}}\right) \left[2 \sqrt{\omega_{1,\Lambda}}
   J_{k+1}\left(\sqrt{\omega_{1,\Lambda}}\right)-(2 k +\omega_{2,\Lambda})
   J_k\left(\sqrt{\omega_{1,\Lambda}}\right)\right]=0.
\end{multline}

For $\mathrm{FD}$ boundary conditions one obtains an analogous formula, which we omit.

One can use Mathematica to find the zeroes of \eqref{secular equation DN} (and the corresponding equation for $\mathrm{FD}$ boundary conditions) numerically and compute the eigenvalue counting function $N_\aleph(\Lambda)$, $\aleph\in\{\mathrm{DF}, \mathrm{FD}\}$, for reasonably large values of the parameter $\Lambda$.

\

The numerical results are shown in Figures~\ref{fig:diskDN} and~\ref{fig:diskND}.

\begin{figure}[htb] 
\begin{center}
\includegraphics[width=\textwidth]{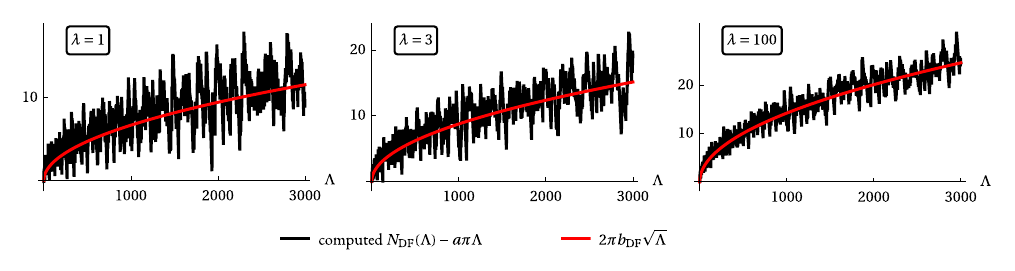}
\caption{The DF eigenvalue problem for the disk. In all images $\mu=1$.}\label{fig:diskDN}
\end{center}
\end{figure}

\begin{figure}[htb] 
\begin{center}
\includegraphics[width=\textwidth]{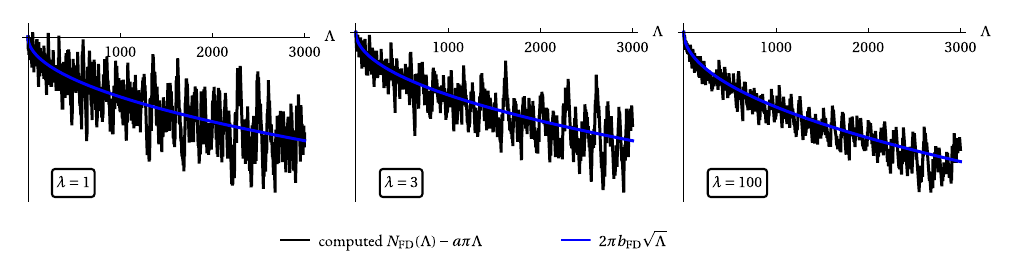}
\caption{The FD eigenvalue problem for the disk. In all images $\mu=1$.}\label{fig:diskND}
\end{center}
\end{figure}

\subsubsection{Flat cylinders}
\label{Flat cylinders 2D}

Consider the two-dimensional cylinder $M:=\mathbb{T}\times [0,h]$, where $\mathbb{T}$ is the one-dimensional torus and $h>0$ is the height of the cylinder, equipped with coordinates $(x^1,x^2)\in [0,2\pi)\times[0,h]$. Of course,
\begin{equation}
\label{volume surface area 2d cylnder}
\operatorname{Vol}_2(M)=2\pi h, \qquad \operatorname{Vol}_1(\partial M)=4\pi.
\end{equation}
We separate variables by seeking a solution in the form
\begin{equation}
\label{solutions 2D cylinder}
\bu(x^1,x^2)=\grad \psi_1(x^1,x^2)+ \curl\left(\psi_2(x^1,x^2)\,\hat{\mathbf{z}}\right)\,,
\end{equation}
where $\hat{\mathbf{z}}$ is the unit vector in the auxiliary coordinate $x^3$ (orthogonal to the $(x^1,x^2)$-plane) and $\psi_j$, $j=1,2$, are scalar potentials. As in subsection~\ref{The disk}, the scalar potentials satisfy Helmholtz equation \eqref{helmholz 2d disk}, \eqref{omegas helmholtz}. The general solution for $\psi_j$, $j=1,2$, reads
\begin{equation}
\label{general solution psi 2d cylinder}
\psi_j(x^1,x^2)=\sum_{\xi\in \mathbb{Z}} \left(
c_{j,\xi,+}e^{i\left(x^1 \xi+\sqrt{\omega_{j, \Lambda}-\xi^2} x^2 \right)} 
+
c_{j,\xi,-}e^{i\left(x^1 \xi-\sqrt{\omega_{j, \Lambda}-\xi^2} x^2 \right)}
\right)\,.
\end{equation}

Substituting \eqref{general solution psi 2d cylinder} into \eqref{solutions 2D cylinder} and, in turn, imposing boundary conditions $\mathcal{B}_{\mathrm{DF}}$ yields the secular equation
\begin{equation}
\label{secular equation 2d cylinder DF}
\Lambda ^2 
   \left(\frac{\Lambda}{\mu} -\xi^2\right) \sin\left(h
   \sqrt{\frac{\Lambda}{\mu
   }-\xi^2}\right) \sin \left(h
   \sqrt{\frac{\Lambda }{\lambda +2
   \mu }-\xi^2}\right)=0\,.
\end{equation}
Similarly, imposing boundary conditions $\mathcal{B}_{\mathrm{FD}}$ yields the secular equation
\begin{equation}
\label{secular equation 2d cylinder FD}
\Lambda ^2 
   \left(\frac{\Lambda}{\lambda+2\mu} -\xi^2\right) \sin\left(h
   \sqrt{\frac{\Lambda}{\mu
   }-\xi^2}\right) \sin \left(h
   \sqrt{\frac{\Lambda }{\lambda +2
   \mu }-\xi^2}\right)=0\,.
\end{equation}

A careful examination of \eqref{solutions 2D cylinder}--\eqref{secular equation 2d cylinder FD} yields the following.

\begin{thm}
\label{theorem DF spectrum 2d cylinder}
The eigenvalues of the Dirichlet-free (DF) eigenvalue problem for the operator of linear elasticity on the two-dimensional cylinder are:
\begin{enumerate}[(i)]
\item Eigenvalues
\begin{equation}
\label{DF 2d series 1}
\frac{k^2\pi^2}{h^2}(\lambda+2\mu), \qquad k=1,2,\dots,
\end{equation}
with multiplicity $1$.
\item Eigenvalues
\begin{equation}
\label{DF 2d series 2}
\frac{k^2\pi^2}{h^2}\mu, \qquad k=1,2,\dots,
\end{equation}
with multiplicity $1$.
\item Eigenvalues
\begin{equation}
\label{DF 2d series 3}
n^2\,\mu, \qquad n=1,2,\dots,
\end{equation}
with multiplicity $2$.
\item Eigenvalues
\begin{equation}
\label{DF 2d series 4}
\left(n^2+\frac{k^2\pi^2}{h^2}\right)\mu, \qquad n,k=1,2,\dots,
\end{equation}
with multiplicity $2$.
\item Eigenvalues
\begin{equation}
\label{DF 2d series 5}
\left(n^2+\frac{k^2\pi^2}{h^2}\right)(\lambda+2\mu), \qquad n,k=1,2,\dots,
\end{equation}
with multiplicity $2$.
\end{enumerate}
\end{thm}

\begin{thm}
\label{theorem FD spectrum 2d cylinder}
The eigenvalues of the free-Dirichlet (FD) eigenvalue problem for the operator of linear elasticity on the two-dimensional cylinder are:
\begin{enumerate}[(i)]
\item Eigenvalues
\begin{equation}
\label{FD 2d series 1}
\frac{k^2\pi^2}{h^2}(\lambda+2\mu), \qquad k=1,2,\dots,
\end{equation}
with multiplicity $1$.
\item Eigenvalues
\begin{equation}
\label{FD 2d series 2}
\frac{k^2\pi^2}{h^2}\mu, \qquad k=1,2,\dots,
\end{equation}
with multiplicity $1$.
\item Eigenvalues
\begin{equation}
\label{FD 2d series 3}
\left(n^2+\frac{k^2\pi^2}{h^2}\right)\mu, \qquad n,k=1,2,\dots,
\end{equation}
with multiplicity $2$.
\item Eigenvalues
\begin{equation}
\label{FD 2d series 4}
n^2(\lambda+2\mu), \qquad n=1,2,\dots,
\end{equation}
with multiplicity $2$.
\item Eigenvalues
\begin{equation}
\label{FD 2d series 5}
\left(n^2+\frac{k^2\pi^2}{h^2}\right)(\lambda+2\mu), \qquad n,k=1,2,\dots,
\end{equation}
with multiplicity $2$.
\end{enumerate}
\end{thm}

\begin{figure}[tb] 
\begin{center}
\includegraphics[width=\textwidth]{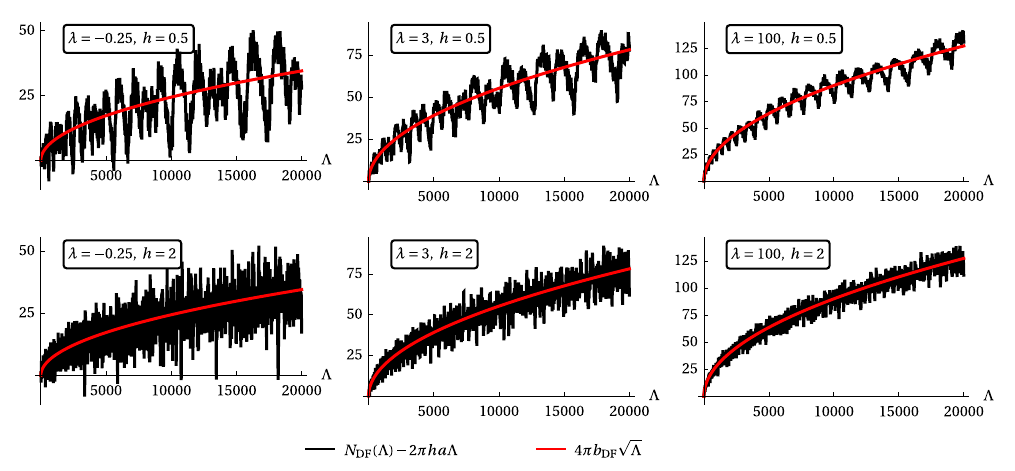}
\caption{The DF eigenvalue problem for 2D flat cylinders. In all images $\mu=1$.}\label{fig:cyl2DN}
\end{center}
\end{figure}

\begin{figure}[tb] 
\begin{center}
\includegraphics[width=\textwidth]{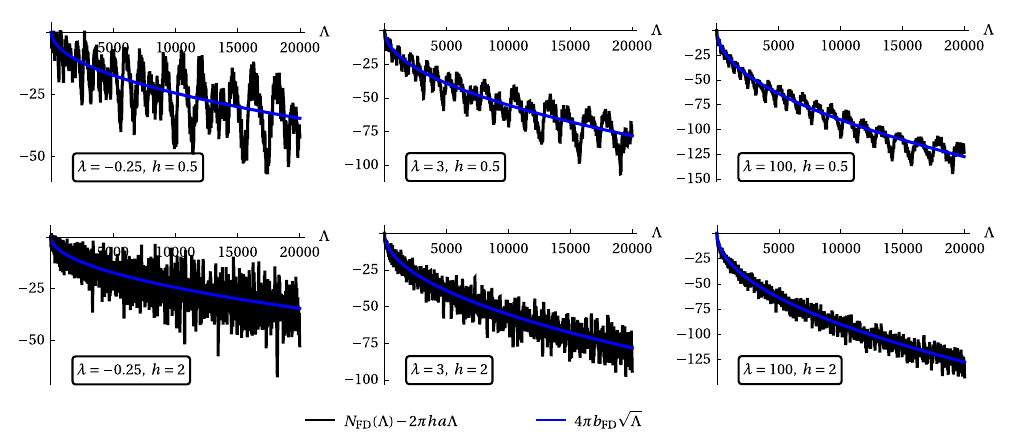}
\caption{The FD eigenvalue problem for 2D flat cylinders. In all images $\mu=1$.}\label{fig:cyl2ND}
\end{center}
\end{figure}

Observe that the DF (Theorem~\ref{theorem DF spectrum 2d cylinder}) and the FD (Theorem~\ref{theorem FD spectrum 2d cylinder}) spectra coincide, except for fact that the series of eigenvalues~\eqref{DF 2d series 3} in the DF spectrum is replaced by the series of eigenvalues~\eqref{FD 2d series 4} in the FD spectrum.

Theorems~\ref{theorem DF spectrum 3d cylinder} and~\ref{theorem FD spectrum 3d cylinder} allow us to write down the eigenvalue counting functions $N_{\mathrm{DF}}$ and $N_{\mathrm{FD}}$ explicitly. They read
\begin{multline}
\label{N DF 2d cylinder}
N_{\mathrm{DF}}(\Lambda)
=
\left\lfloor \frac{h}{\pi}\sqrt{\frac{\Lambda}{\lambda+2\mu}} \right\rfloor
+
\left\lfloor \frac{h}{\pi}\sqrt{\frac{\Lambda}{\mu}} \right\rfloor
+
2\left\lfloor \sqrt{\frac{\Lambda}{\mu}} \right\rfloor
\\
+
2\sum_{n=1}^{\left\lfloor \sqrt{\frac{\Lambda}{\mu}}\right\rfloor} 
 \left\lfloor\frac{h}{\pi}\sqrt{\frac{\Lambda}{\mu}- n^2} \right\rfloor
+
2\sum_{n=1}^{\left\lfloor \sqrt{\frac{\Lambda}{\lambda+2\mu}} \right\rfloor} \left\lfloor\frac{h}{\pi}\sqrt{\frac{\Lambda}{\lambda+2\mu}- n^2} \right\rfloor,
\end{multline}

\begin{equation}
\label{N FD 2d cylinder}
N_{\mathrm{FD}}(\Lambda)
=
N_{\mathrm{DF}}(\Lambda)+2\left( \left\lfloor \sqrt{\frac{\Lambda}{\lambda+2\mu}} \right\rfloor-\left\lfloor \sqrt{\frac{\Lambda}{\mu}} \right\rfloor\right)\,.
\end{equation}
Here $\lfloor\,\cdot\, \rfloor$ denotes the integer part (floor function).

\

Let us verify formula~\eqref{main theorem equation 1} by computing the asymptotic expansions of~\eqref{N DF 2d cylinder} and~\eqref{N FD 2d cylinder} as $\Lambda \to +\infty$.

\begin{prop}
\label{prop asymptotic expansions 2d cylinder}
The functions \eqref{N DF 2d cylinder} and \eqref{N FD 2d cylinder} admit the following two-term asymptotic expansion:
\begin{equation}
\label{prop asymptotic expansions 2d cylinder equation 1}
N_{\aleph}(\Lambda)=\frac{h}2 \left(\frac{1}{\mu}+\frac{1}{\lambda+2\mu}\right) \Lambda
\pm
\left(\frac{1}{\mu^{1/2}}-\frac{1}{(\lambda+2\mu)^{1/2}}\right) \Lambda^{1/2}+o(\Lambda^{1/2}) \quad \text{as}\quad \Lambda\to +\infty
\end{equation}
with sign
\begin{equation}
\label{prop asymptotic expansions 2d cylinder equation 2}
\begin{cases}
+ & \text{for}\ \,\aleph=\mathrm{DF},\\
- & \text{for}\ \,\aleph=\mathrm{FD}.
\end{cases}
\end{equation}
\end{prop}
\begin{proof}
Formula~\eqref{prop asymptotic expansions 2d cylinder equation 1} follows from \eqref{N FD 2d cylinder}, \eqref{N DF 2d cylinder}, and the estimate
\begin{equation*}
\label{tricky estimate 2d}
\sum_{n=1}^{\lfloor \sqrt{x}\rfloor} \left\lfloor a \sqrt{x-n^2} \right\rfloor=\frac{\pi a}{4}x-\frac12\left(a+1 \right) x^{1/2}+o(x^{1/2})\quad \text{as}\quad x\to +\infty\,, \quad a>0\,.
\end{equation*}
\end{proof}

On account of~\eqref{volume surface area 2d cylnder}, Proposition~\ref{prop asymptotic expansions 2d cylinder} agrees with Theorem~\ref{main theorem} as well as formula~\eqref{FD less DF}.

\


Figures~\ref{fig:cyl2DN} and~\ref{fig:cyl2ND} show a comparison between the actual counting functions~\eqref{N DF 2d cylinder}, \eqref{N FD 2d cylinder} and the two-term asymptotic expansions~\eqref{prop asymptotic expansions 2d cylinder equation 1}, \eqref{prop asymptotic expansions 2d cylinder equation 2}.

\subsection{Three-dimensional examples}
\label{Three-dimensional examples}

\subsubsection{Flat cylinders}
\label{Flat cylinders 3D}

Consider the three-dimensional cylinder $M:=\mathbb{T}^2\times [0,h]$, where $\mathbb{T}^2$ is the flat two-dimensional torus with side $2\pi$ and $h>0$ is the height of the cylinder, equipped with coordinates $(x^1,x^2,x^3)\in [0,2\pi)^2\times[0,h]$. Of course,
\begin{equation}
\label{volume surface area 3d cylnder}
\operatorname{Vol}_3(M)=4\pi^2 h, \qquad \operatorname{Vol}_2(\partial M)=8\pi^2.
\end{equation}
We separate variables by seeking a solution in the form
\begin{equation}
\label{solutions 3D cylinder}
\bu(x^1,x^2,x^3)=\grad \psi_1(x^1,x^2,x^3)+ \curl\left(\psi_2(x^1,x^2,x^3)\,\hat{\mathbf{z}}\right)+\curl\curl\left(\psi_3(x^1,x^2,x^3)\,\hat{\mathbf{z}}\right)\,,
\end{equation}
where $\hat{\mathbf{z}}$ is the unit vector in the (positive) direction $x^3$ and $\psi_j$, $j=1,2,3$, are scalar potentials. Once again, the scalar potentials satisfy Helmholtz equation \eqref{helmholz 2d disk}, with $\omega_{1,\Lambda}$ and $\omega_{2, \Lambda}$ defined in accordance with \eqref{omegas helmholtz}, and $\omega_{3, \Lambda}:=\omega_{2, \Lambda}$. The general solution for $\psi_j$, $j=1,2,3$, reads
\begin{equation}
\label{general solution psi 3d cylinder}
\psi_j(x^1,x^2,x^3)=\sum_{(\xi_1,\xi_2)\in \mathbb{Z}^2} \left(
c_{j,\xi_1,\xi_2,+}e^{i\left(x^1 \xi_1+x^2 \xi_2+\sqrt{\omega_{j, \Lambda}-n} x^3 \right)} 
+
c_{j,\xi_1,\xi_2,-}e^{i\left(x^1 \xi_1+x^2 \xi_2+\sqrt{\omega_{j, \Lambda}-n} x^2 \right)}
\right)\,,
\end{equation}
where $n:=\xi_1^2+\xi_2^2$.
Substituting \eqref{general solution psi 3d cylinder} into \eqref{solutions 3D cylinder} and, in turn, imposing boundary conditions $\mathcal{B}_{\mathrm{DF}}$ at $x^3=0$ and $x^3=h$ yields the secular equation
\begin{equation}
\label{secular equation 3d cylinder DF}
\Lambda ^2 n^2
   \left(\frac\Lambda\mu -  n\right) \sin ^2\left(h
   \sqrt{\frac{\Lambda}{\mu
   }-n}\right) \sin \left(h
   \sqrt{\frac{\Lambda }{\lambda +2
   \mu }-n}\right)=0\,.
\end{equation}
Similarly, imposing boundary conditions $\mathcal{B}_{\mathrm{FD}}$ yields the secular equation
\begin{equation}
\label{secular equation 3d cylinder FD}
\Lambda ^2 n^2
   \left(\frac\Lambda\mu -  n\right) \left(\frac{\Lambda}{\lambda+2\mu} -n\right) \sin ^2\left(h
   \sqrt{\frac{\Lambda}{\mu
   }-n}\right) \sin \left(h
   \sqrt{\frac{\Lambda }{\lambda +2
   \mu }-n}\right)=0\,.
\end{equation}

The $\mathrm{DF}$ (resp.~$\mathrm{FD}$) spectrum is a subset of the zeroes of \eqref{secular equation 3d cylinder DF} (resp.~\eqref{secular equation 3d cylinder FD}). A direct examination of solutions of~\eqref{secular equation 3d cylinder DF} and~\eqref{secular equation 3d cylinder FD} yields the following.

Let $r_2:\mathbb{N} \to \mathbb{N}$ be the sum of squares function:
\begin{equation*}
r_2(n):=\#\left\{(a,b)\in \mathbb{Z}^2\ | \ n=a^2+b^2\right\}\,.
\end{equation*}

\begin{thm}
\label{theorem DF spectrum 3d cylinder}
The eigenvalues of the Dirichlet-free eigenvalue problem for the operator of linear elasticity on the three-dimensional flat cylinder are:
\begin{enumerate}[(i)]
\item Eigenvalues
\begin{equation}
\label{DF series 1}
\frac{k^2\pi^2}{h^2}(\lambda+2\mu), \qquad k=1,2,\dots,
\end{equation}
with multiplicity $1$.
\item Eigenvalues
\begin{equation}
\label{DF series 2}
n\,\mu, \qquad n=1,2,\dots,
\end{equation}
with multiplicity $r_2(n)$.
\item Eigenvalues
\begin{equation}
\label{DF series 3}
\left(n+\frac{k^2\pi^2}{h^2}\right)\mu, \qquad n,k=1,2,\dots,
\end{equation}
with multiplicity $2r_2(n)$.
\item Eigenvalues
\begin{equation}
\label{DF series 4}
\left(n+\frac{k^2\pi^2}{h^2}\right)(\lambda+2\mu), \qquad n,k=1,2, \dots, 
\end{equation}
with multiplicity $r_2(n)$.
\end{enumerate}
\end{thm}
\begin{thm}
\label{theorem FD spectrum 3d cylinder}
The eigenvalues of the free-Dirichlet eigenvalue problem for the operator of linear elasticity on the three-dimensional flat cylinder are:
\begin{enumerate}[(i)]
\item Eigenvalues
\begin{equation}
\label{FD series 1}
\frac{k^2\pi^2}{h^2}(\lambda+2\mu), \qquad k=1,2,\dots,
\end{equation}
with multiplicity $1$.
\item Eigenvalues
\begin{equation}
\label{FD series 2}
n\,\mu, \qquad n=1,2,\dots,
\end{equation}
with multiplicity $r_2(n)$\footnote{Here and further on by multiplicity zero we mean that the corresponding number is not an eigenvalue.}.
\item Eigenvalues
\begin{equation}
\label{FD series 3}
\left(n+\frac{k^2\pi^2}{h^2}\right)\mu, \qquad n,k=1,2,\dots,
\end{equation}
with multiplicity $2r_2(n)$.
\item Eigenvalues
\begin{equation}
\label{FD series 4}
n\,(\lambda+2\mu), \qquad n=1,2,\dots,
\end{equation}
with multiplicity $r_2(n)$.
\item Eigenvalues
\begin{equation}
\label{FD series 5}
\left(n+\frac{k^2\pi^2}{h^2}\right)(\lambda+2\mu), \qquad n,k=1,2, \dots, 
\end{equation}
with multiplicity $r_2(n)$.
\end{enumerate}
\end{thm}

Observe that the DF (Theorem~\ref{theorem DF spectrum 3d cylinder}) and the FD (Theorem~\ref{theorem FD spectrum 3d cylinder}) spectra coincide, except for the additional series of eigenvalues~\eqref{FD series 4} in the FD spectrum.

Theorems~\ref{theorem DF spectrum 3d cylinder} and~\ref{theorem FD spectrum 3d cylinder} allow us to write down the eigenvalue counting functions $N_{\mathrm{DF}}$ and $N_{\mathrm{FD}}$ explicitly. They read
\begin{multline}
\label{N DF 3d cylinder}
N_{\mathrm{DF}}(\Lambda)
=
\left\lfloor \frac{h}{\pi}\sqrt{\frac{\Lambda}{\lambda+2\mu}} \right\rfloor
+
\sum_{n=1}^{\left\lfloor \frac{\Lambda}{\mu}\right\rfloor} r_2(n) \left( 2\left\lfloor\frac{h}{\pi}\sqrt{\frac{\Lambda}{\mu}- n} \right\rfloor+1\right)
\\+
\sum_{n=1}^{\left\lfloor \frac{\Lambda}{\lambda+2\mu} \right\rfloor} r_2(n) \left\lfloor\frac{h}{\pi}\sqrt{\frac{\Lambda}{\lambda+2\mu}- n} \right\rfloor \,,
\end{multline}
\begin{equation}
\label{N FD 3d cylinder}
N_{\mathrm{FD}}(\Lambda)
=
N_{\mathrm{DF}}(\Lambda)
+
\sum_{n=1}^{\left\lfloor \frac{\Lambda}{\lambda+2\mu} \right\rfloor} r_2(n).
\end{equation}

Let us verify formula~\eqref{main theorem equation 1} by computing the asymptotic expansions of~\eqref{N FD 3d cylinder} and~\eqref{N DF 3d cylinder} as $\Lambda \to +\infty$.

\begin{prop}
\label{prop asymptotic expansions 3d cylinder}
The functions \eqref{N FD 3d cylinder} and \eqref{N DF 3d cylinder} admit the following two-term asymptotic expansion:
\begin{equation}
\label{prop asymptotic expansions 3d cylinder equation 1}
N_{\aleph}(\Lambda)=\frac{2h}3 \left(\frac{2}{\mu^{3/2}}+\frac{1}{(\lambda+2\mu)^{3/2}}\right) \Lambda^{3/2}
\mp
\frac{\pi}{2(\lambda+2\mu)}\Lambda+o(\Lambda) \quad \text{as}\quad \Lambda\to +\infty
\end{equation}
with sign
\begin{equation}
\label{prop asymptotic expansions 3d cylinder equation 2}
\begin{cases}
- & \text{for}\ \,\aleph=\mathrm{DF},\\
+ & \text{for}\ \,\aleph=\mathrm{FD}.
\end{cases}
\end{equation}
\end{prop}
\begin{proof}
Formula~\eqref{prop asymptotic expansions 3d cylinder equation 2} follows from \eqref{N FD 3d cylinder}, \eqref{N FD 3d cylinder}, and the estimates
\begin{equation*}
\label{gauss circle problem}
\sum_{n=1}^{\lfloor x \rfloor} r_2(n)=\pi x+ O(x^{1/3}) \quad \text{as}\quad x\to+\infty,
\end{equation*}
\begin{equation*}
\label{tricky estimate}
\sum_{n=1}^{\lfloor x \rfloor}\left\lfloor a \sqrt{x-n} \right\rfloor r_2(n)=\frac{2\pi a}{3}x^{3/2}-\frac\pi2x+ o(x) \quad \text{as}\quad x\to+\infty, \quad a>0.
\end{equation*}
\end{proof}

On account of~\eqref{volume surface area 3d cylnder}, Proposition~\ref{prop asymptotic expansions 3d cylinder} agrees with Theorem~\ref{main theorem} as well as formula~\eqref{DF less FD}.

\

Figures~\ref{fig:cyl3DN} and~\ref{fig:cyl3ND} show a comparison between the actual counting functions~\eqref{N DF 3d cylinder}, \eqref{N FD 3d cylinder} and the two-term asymptotic expansions~\eqref{prop asymptotic expansions 3d cylinder equation 1}, \eqref{prop asymptotic expansions 3d cylinder equation 2}.

\begin{figure}[htb] 
\begin{center}
\includegraphics[width=\textwidth]{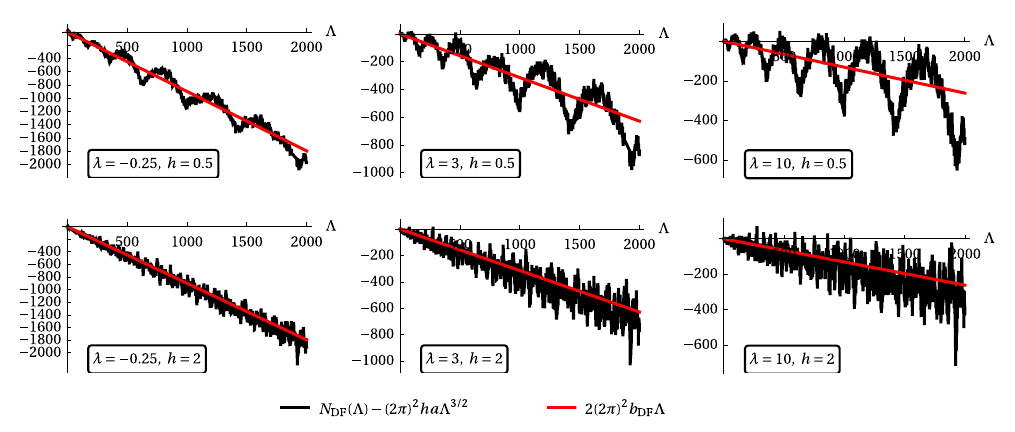}
\caption{The DF eigenvalue problem for 3D flat cylinders. In all images $\mu=1$.}\label{fig:cyl3DN}
\end{center}
\end{figure}

\begin{figure}[htb] 
\begin{center}
\includegraphics[width=\textwidth]{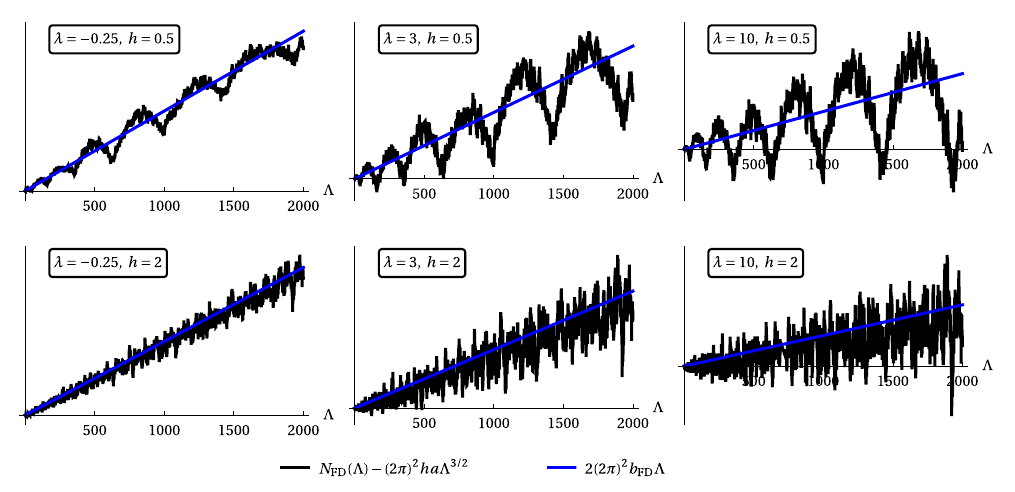}
\caption{The FD eigenvalue problem for 3D flat cylinders. In all images $\mu=1$.}\label{fig:cyl3ND}
\end{center}
\end{figure}

\section*{Acknowledgements}
\addcontentsline{toc}{section}{Acknowledgements}

We are indebted to Michael Levitin, Yiannis Petridis and Dmitri Vassiliev for insightful conversations on aspects of this paper, and to Gerd Grubb and Grigori Rozenblum for useful bibliographic suggestions.

MC was partially supported by a grant of the Heilbronn Institute for Mathematical Research (HIMR) via the UKRI/EPSRC and by EPSRC Fellowship EP/X01021X/1. 
IM was supported by a MAC-MIGS Summer Internship (MAC-MIGS CDT, Maxwell Institute Graduate School, Edinburgh).

\begin{appendices} 
%

\end{appendices}

\end{document}